\documentclass[english, a4paper, 11pt]{article}

\usepackage[margin=1in]{geometry}

\usepackage{tikz}
\usepackage[utf8]{inputenc}
\usepackage[T1]{fontenc}
\usepackage{amsmath,amsthm,amsfonts,amssymb,amscd}
\usepackage{thmtools}
\usepackage{bbold}
\usepackage[dvipsnames]{xcolor}
\usepackage{mathrsfs}
\usepackage{hyperref}
\usepackage{caption}
\usepackage{subcaption}
\theoremstyle{plain}
\newtheorem{theorem}{Theorem}[section]
\theoremstyle{plain}
\usepackage{mathtools}
\newcommand{\Z}{\mathbb Z}
\newcommand{\R}{\mathbb R}
\newcommand{\A}{\mathscr A}

\newcommand{\E}{\mathbb E}
\newcommand{\N}{\mathbb N}
\newcommand{\Pcal}{\mathcal P}

\renewcommand{\P}{\mathbb P}
\newcommand{\W}{E_{\rm co}}
\newcommand{\e}{\varepsilon}
\newcommand{\sgn}{\text{\normalfont sgn}}
\newcommand{\C}{\mathcal C}
\newcommand{\F}{\mathcal F}

\newtheorem{lemma}[theorem]{Lemma}
\newtheorem{proposition}[theorem]{Proposition}
\newtheorem{remark}[theorem]{Remark}

\numberwithin{equation}{section}
\theoremstyle{definition}
\newtheorem{definition}[theorem]{Definition}
\usepackage{cleveref}

\newcommand{\inde}{indistinguishable}
\newcommand{\indy}{indistinguishability}
\newcommand{\lora}{\longrightarrow}
\newcommand{\llra}{\longleftrightarrow}

\title{Indistinguishability for recurrent clusters}
\author{Damis El Alami\footnote{Department of Stochastics, Institute of Mathematics, Budapest University of Technology and Economics, M\H{u}egyetem rkp.~3., Budapest 1111 Hungary}\ \footnote{HUN-REN Alfr\'ed R\'enyi Institute of Mathematics, Re\'altanoda u.~13-15, Budapest 1053 Hungary} 
\qquad Gábor Pete\footnotemark[2]\ \footnotemark[1] 
\qquad Ádám Timár \footnotemark[2]\ \footnote{Division of Mathematics, The Science Institute, University of Iceland, Dunhaga 3 IS-107 Reykjavík, Iceland} 
}
\date{\today}

\begin{document}
\maketitle

\begin{abstract}
We introduce a general framework to show the indistinguishability of infinite clusters (ergodicity of the cluster subrelation) in group-invariant percolation processes with a weaker version of the finite energy property: the possibility of moving infinite branches from one infinite cluster to another. Crucially, this removes the necessity for the infinite clusters to be transient, present in most previous works. Our method also applies to more general random graphs, whenever a stationary sequence of vertices is definable.

We use this to show the indistinguishability of infinite clusters (or permutation cycles) in the interchange process (a.k.a.~random stirring process), the loop $O(n)$ model on amenable Cayley graphs, biased corner percolation on $\Z^2$, and the Poisson Zoo process.

Finally, we show that infinite clusters in any invariant process on a Cayley graph are indistinguishable for any ``not essentially tail'' property, i.e., properties that depend only on the local structure of the cluster.
\end{abstract}

\tableofcontents

\section{Introduction and results}

Consider a Cayley graph $G=G(\Gamma,S)$ of a group $\Gamma$ with a finite generating set $S$, and a $\Gamma$-translation-invariant probability measure on subsets of the vertices $V(G)=\Gamma$ or the edges $E(G)$, called a vertex- or edge-percolation on $G$. These are key examples of probability measure preserving (p.m.p.) group actions, fundamental from ergodic theoretical  \cite{ornstein1987entropy,kechris2004topics,bowen2020examples} and probabilistic and statistical physical \cite{benjamini1996percolation, benjamini1999group,LPbook,PGG} points of view, as well.

Two connected components, called clusters, are called {\bf \inde} if, for any $\Gamma$-invariant property for a cluster (e.g., cardinality, or recurrence of simple random walk on the cluster itself, or whether it is almost surely hit by random walk on $G$ started at any vertex), called \textbf{component property}, either both clusters have it, or neither. 
Finite clusters are of course distinguishable by cardinality, but how about infinite clusters? Lyons and Schramm \cite{lyonsIndistinguishabilityPercolationClusters1999}, who introduced this concept, proved for invariant percolations on Cayley graphs, or more generally, unimodular transitive graphs, that {\bf insertion tolerance} (which means, roughly, that one can add vertices or edges to the configuration with a positive conditional probability, given any configuration elsewhere) implies that all infinite clusters are a.s.~\inde. It is important to note here that non-uniqueness of infinite clusters in an insertion tolerant automorphism-invariant percolation is possible only on non-amenable transitive graphs \cite{burtonkeane}, hence this Lyons-Schramm theorem is non-vacuous only in the non-amenable setting.

Indistinguishability is an extremely natural question: it is the {\bf extremality} (ergodicity) of the unimodular random rooted graph \cite{aldousProcessesUnimodularRandom2007} that is the cluster of a fixed vertex as the root, together with its supergraph $G$ (called local unimodularity in \cite{hutchcroft2020non}). In a slightly stronger notion of \indy, we have the supergraph $G$ labelled not only by the cluster of the root, but by the entire percolation configuration so that we can talk about \indy\ of the clusters together with the environment, asking about properties like ``the cluster of the root touches every other cluster at infinitely many places''. (This is the notion we will mean by \indy\ below, defined precisely in Subsection~\ref{ss:defs}). This stronger notion was noticed to be equivalent to the {\bf ergodicity of the cluster equivalence relation} in~\cite{gaboriau2009measurable}, hence it is very natural also in the orbit equivalence world.

Moreover, \indy\ has turned out to have many important applications. Already in \cite{lyonsIndistinguishabilityPercolationClusters1999}, applications to Bernoulli percolation included a new, conceptual, proof of {\bf uniqueness monotonocity} \cite{haggstrom1999percolation} in the case of unimodular transitive graphs: there exist $0<p_c \leq p_u\leq 1$ such that for any $p<p_c$ there are only finite clusters, for $p \in (p_c,p_u)$ there are infinitely many infinite clusters, and for $p>p_u$ there is a unique infinite cluster, almost surely. Some of these phases might be empty in general, but one example corollary in the paper is that Kazhdan (T) groups \cite{glasner1997kazhdan} have non-uniqueness at $p_u$, hence $p_u<1$. Another application is the {\bf continuity} of the percolation probability $\theta(p):=\P_p(o\llra\infty)$ for $p>p_c$, from the argument in~\cite{vdBK84}. 

Later, \indy\ of the trees in the {\bf Free and Wired Uniform Spanning Forests} and the {\bf Free Minimal Spanning Forest} (when it is not equal to the Wired) was proved, see~\cite{hutchcroft2017indistinguishability, timar2006ends, timar2018indistinguishability}. Spanning forests obviously do not possess insertion tolerance. Instead, weaker forms (weak insertion tolerance \cite{timar2018indistinguishability} and update tolerance \cite{hutchcroft2017indistinguishability}) were established and used. These \indy\ results have also been used many times since then, e.g., to understand the {\bf geometry of the USF clusters} in $\Z^d$ \cite{hutchcroft2019component}, and in the positive resolution of the {\bf dynamical von Neumann-Day problem} by Gaboriau-Lyons~\cite{gaboriau2009measurable}, generalized in \cite{gheysens2017fixed}: for the Bernoulli shift $\big( [0,1]^\Gamma,\otimes_\Gamma \mathsf{Leb} \big)$ of any countable non-amenable group $\Gamma$, there is an ergodic free p.m.p.\ action of the free group $\mathbf{F}_2$ such that almost every $\Gamma$-orbit decomposes into $\mathbf{F}_2$-orbits. 

 A further important development was that \cite{chifan2010ergodic}, using von Neumann algebras, proved that, in any {\bf factor of i.i.d.~percolation}, there are at most countably many {\bf non-hyperfinite}  \indy\ classes, where non-hyperfiniteness of a cluster means, roughly, that there exists a $\delta>0$ such that if one further removes less than $\delta$ density of the edges in a unimodular way, then the cluster cannot fall apart into finite components only. They also proved that for non-hyperfinite FIID clusters, indistinguishability, i.e., ergodicity of the cluster relation, is equivalent to strong ergodicity; the corresponding notion of {\bf strong indistinguishability} was defined in \cite{martineau2016ergodicity}. See \cite{qCI} and \cite{Mikolaj} for a recent strengthening of the Chifan-Ioana theorem for exact groups, and its relation to {\it sparse} non-hyperfinite FIID clusters.

All the above results, apart from the case of the Wired Uniform Spanning Forest, only make sense in non-amenable graphs, for non-hyperfinite clusters. In fact, a key step in~\cite{lyonsIndistinguishabilityPercolationClusters1999} involves the {\bf transience} of the infinite clusters, which comes from the fact that they are $\infty$-ended. (The {\bf number of ends} of a graph $H$ is the supremum of the number of connected components of $H\setminus S$, with $S$ ranging over all finite subsets of $V(H)$. For a unimodular random rooted graph, the number of ends is 1, 2, or infinite \cite{benjamini1999group}.) 
%because the infinite clusters are \textbf{$\infty$-ended}.
%If one removes a finite set from an infinite cluster, there remains at least one infinite connected component, which we call a branch.
%A graph is said to be \textbf{$n$-ended} if there exists a family of $n$ pairwise disjoint branches, and no family of $n+1$ pairwise disjoint branches.
%It is $\infty$-ended if one can find such a family for any $n$.
There are however many natural models where the infinite clusters are 1- or 2-ended and recurrent. In the present work, we give a general framework to prove indistinguishability, similar but more versatile than \cite{lyonsIndistinguishabilityPercolationClusters1999}, and use it to show the main general tool of this article, \Cref{cor: technical on graphs}.

This proposition essentially says that if one can take a branch from an infinite cluster and graft it onto another with positive probability, then the grafted cluster is indistinguishable from the original owner of the branch.
In other words, it is enough to know one branch of a cluster to know anything measurable and invariant about it with almost sure certainty.

We now give several models where we can apply our method to show indistinguishability.
%The proofs and details will be dealt with in later sections.

First, we take a look at three models with two-ended clusters.
Our most straightforward example is the {\bf interchange process}, a random permutation model that was introduced in \cite{harris1972nearest}, with its physical relevance discovered in \cite{toth1993improved}. Take any Cayley graph $G$, and to every edge $e$, assign an independent unit rate Poisson process $\psi_e$ of clock rings, and each time a clock rings, swap the two particles at the endpoints of $e$, stopping at some fixed time $\beta>0$. At this point, the particles have been permuted by a random permutation $\pi_\beta$. It has been conjectured by B\'alint T\'oth \cite{toth1993improved} that, whenever $G$ is transient for simple random walk, for all large enough $\beta$, the permutation $\pi_\beta$ has infinite cycles almost surely. This was shown in \cite{angel2003random,hammondInfiniteCyclesRandom2013,hammond2015sharp} for regular trees $\mathbb{T}_d$, $d\ge 3$, and recently for $G=\Z^d$, $d\geq 5$, in \cite{elboimInfiniteCyclesInterchange2024}.

\begin{theorem}\label{t.interchange}
The interchange process on any Cayley graph $G$ %together with its environment 
 has \inde\ infinite cycles (together with the environment of the clocks $(\psi_e)$).
\end{theorem}

For instance, if any of the cycles on $\Z^d$ had an asymptotic direction, then all of them would have the same direction, which is impossible by the symmetries and ergodicity of the model. (Note that it is also shown in \cite{elboimInfiniteCyclesInterchange2024} that infinite cycles can globally be approximated by Brownian motions, which is of course much stronger than having no asymptotic direction.)

\begin{figure}[hbtp]
    \centering
    \begin{subfigure}[t]{0.45\textwidth}
    \def\svgwidth{\textwidth}
    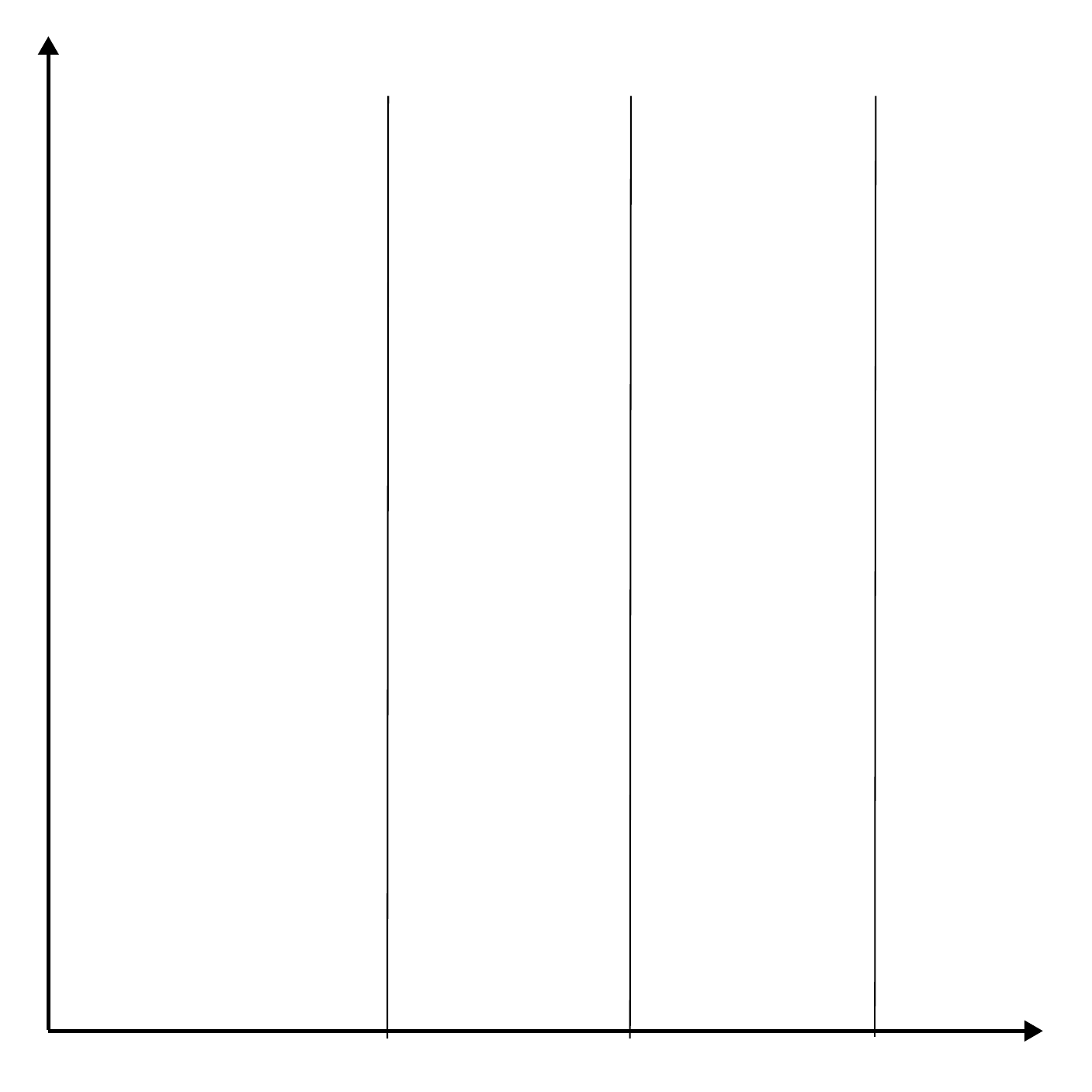
%    \caption{The interchange process on a path graph of length 4, in red the trajectory of particle 1.}
    \end{subfigure}
    \hfill
    \begin{subfigure}[t]{0.45\textwidth}
        \includegraphics[width=\textwidth]{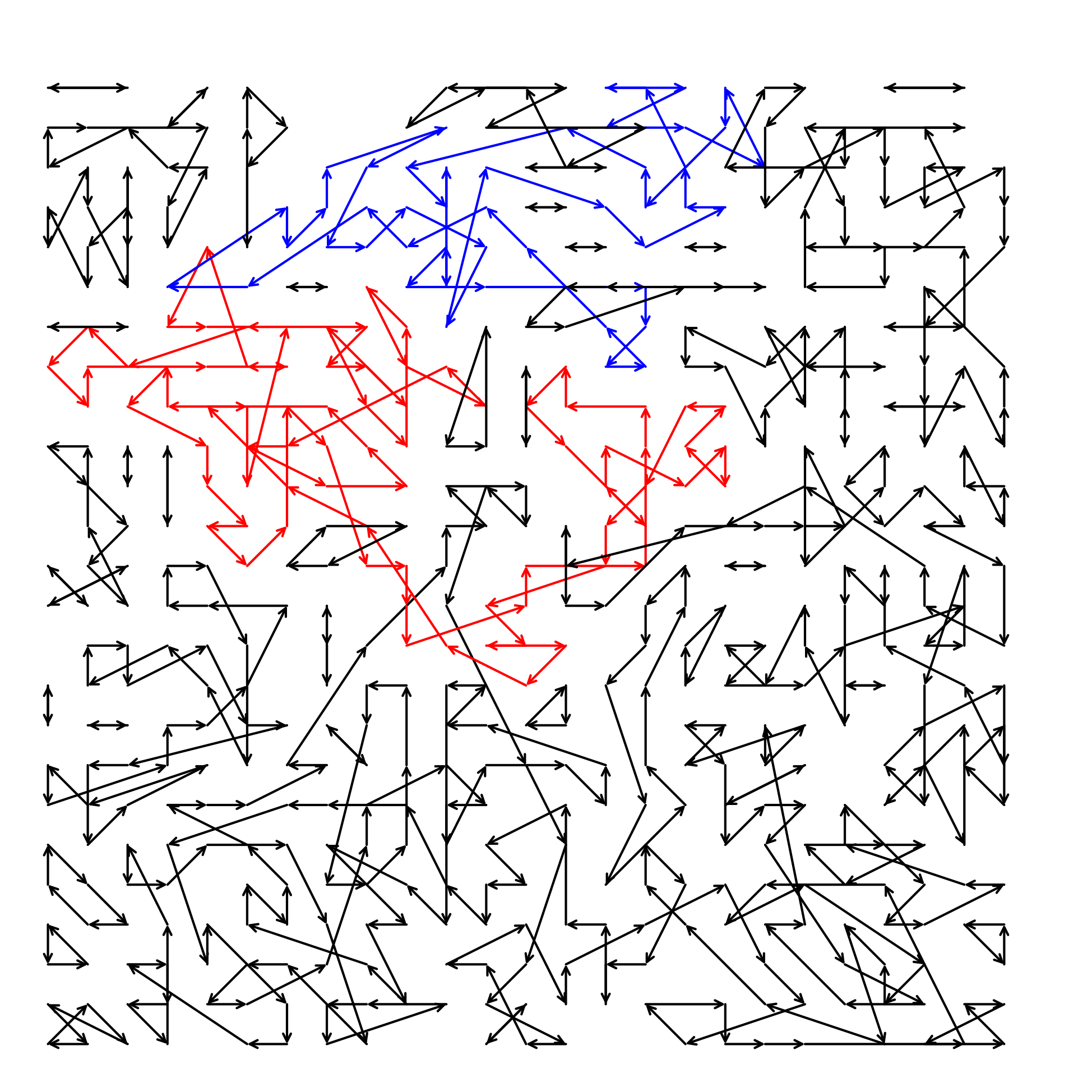}
%    \caption{The three longest cycles in the interchange process on $[\![1,25]\!]^2$.}
    \end{subfigure}
\caption{On the left, the interchange process on a path of length 4; in red the trajectory of particle 1. On the right, the two longest cycles in the interchange process on $[\![1,25]\!]^2$ with $\beta=1$.}
\end{figure}

\Cref{cor: technical on graphs} does not literally apply here, as this is a permutation model, not a percolation, but the proof is very close to the two-ended case in the proof of that proposition. In the following two percolation models, on the other hand, the proposition applies directly.
%have the same structure of two-ended cycles, and indistinguishability can be deduced directly from~\Cref{cor: technical on graphs}.

The first one is the {\bf loop $O(n)$ model}.
This is a random model for disjoint loops and bi-infinite paths on a graph, parametrized by a loop weight $n\geq 0$ and an edge weight $x\geq 0$.
In finite volume, the probability of a configuration is proportional to $x^{|\omega|}n^{l(\omega)}$ where $l(\omega)$ is the number of loops and $|\omega|$ is the total number of edges in all loops.
It is natural and of particular interest on the hexagonal lattice, on which it is equivalent to a range of well-studied models, including self-avoiding walk for $n=0$, the Ising model for $n=1$, and proper 4-colorings for $n=2$ and $x=\infty$;  see~\cite{peled2019lecturesspinloopon, duminilcopin2016exponentialdecaylooplengths}. On infinite graphs, one can take an exhaustion by finite subsets, with certain boundary conditions, and consider weak limits, or at least subsequential limit measures. As far as we know, there are no proved examples with infinitely many infinite cycles, but similarly to the interchange process, it is conjectured that typically this is the case. For the best current results for $\Z^d$, $d\ge 3$, see \cite{QuiTaggi}. 
Let us note that \indy\ and our surgery technique may actually be useful for understanding the number of infinite clusters in some scenarios.  

\begin{theorem}
    \label{thm: O(n) indistinguishability}
    On any amenable Cayley graph $G$, in any limit point of the loop $O(n)$ measure for any $x,n \ge 0$, the infinite components are indistinguishable.
\end{theorem}
\begin{figure}[hbtp]
    \centering
    \begin{subfigure}{0.49\textwidth}
    \includegraphics[width=\textwidth]{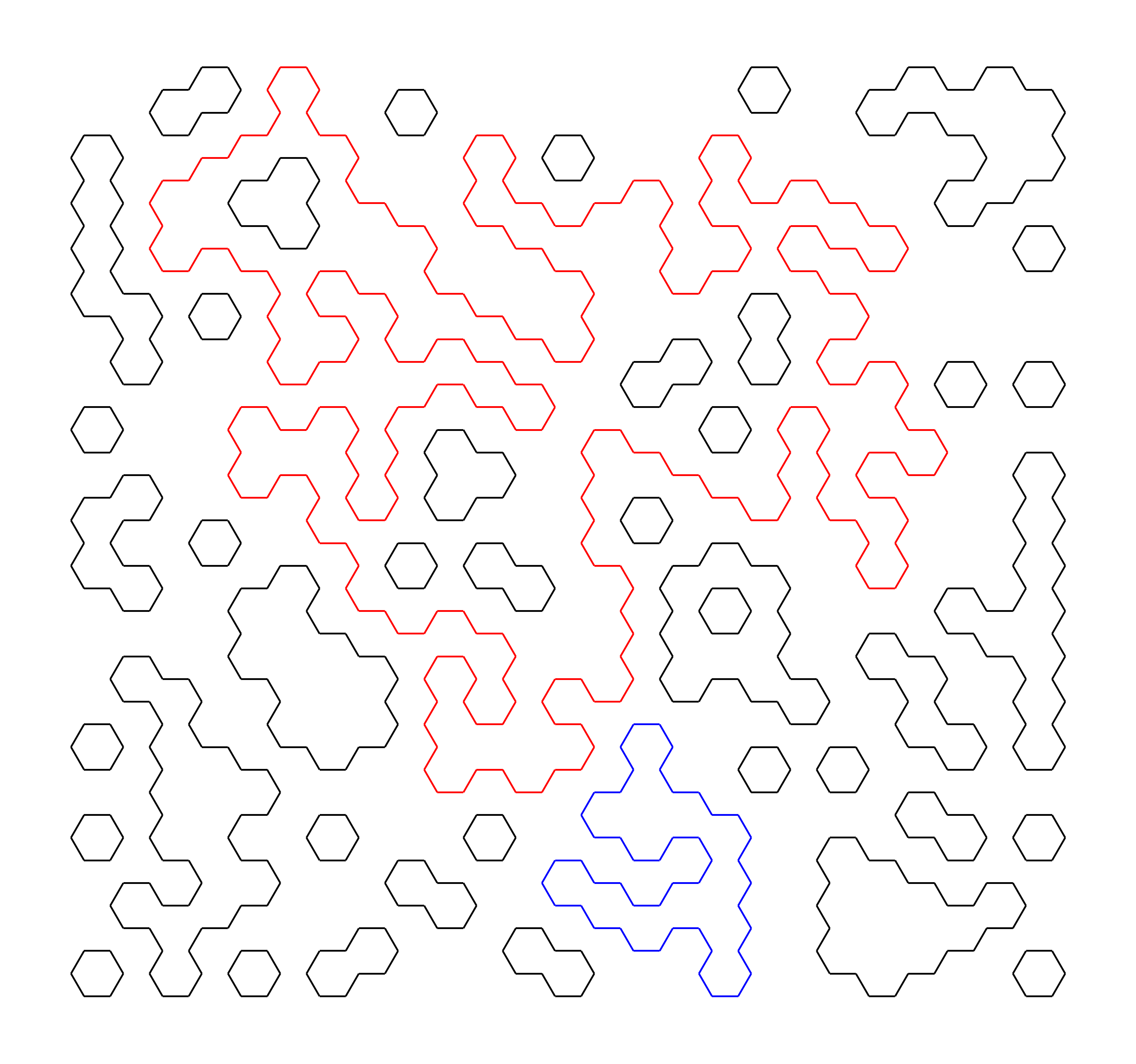}
    \end{subfigure}
    \begin{subfigure}{0.49\textwidth}
        \includegraphics[width=\textwidth]{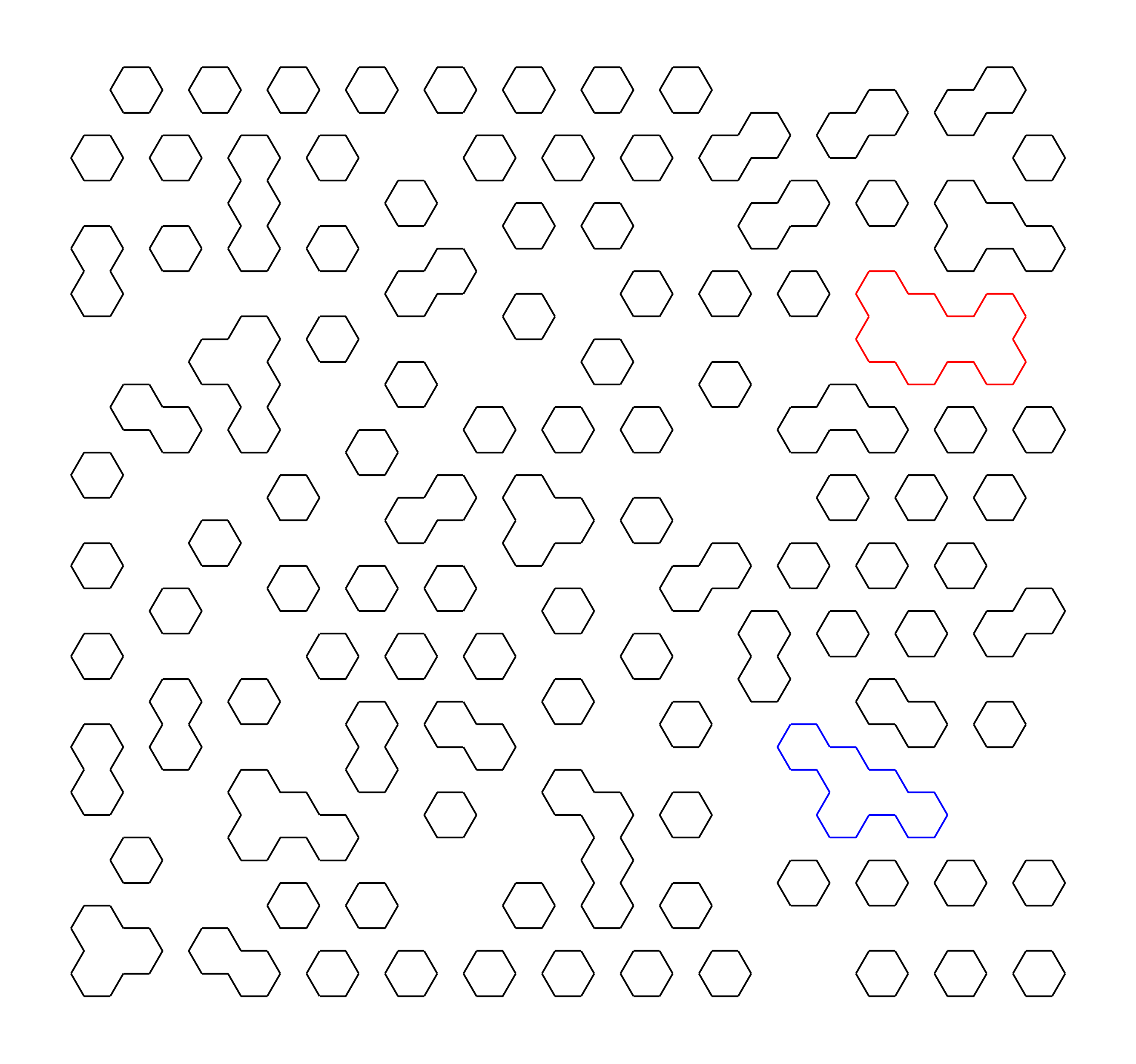}
    \end{subfigure}
    \caption{Two loop $O(n)$ samples in finite volume, $n=2$ on the left, $n=10$ on the right. These pictures are only meant to give an idea of the model in the simplest case, the two-dimensional hexagonal lattice, but here the infinite volume measure is not believed to have infinite cycles.}
\end{figure}

The next one is {\bf corner percolation on $\Z^2$} \cite{peteCornerPercolationMathbb2008a}, quite incidentally introduced by B\'alint T\'oth again, independently considered also in \cite{hooper2013renormalization} in relation with polygon exchange maps, and in \cite{seaton2023mathematical} as recreational math. This is a 2-regular random subgraph of $\Z^2$, where at each vertex we see one of the four possible ``corners''. It can also be considered as a degenerate double dimer model \cite{kenyon2014conformal}. The fully symmetric unbiased case was proved in~\cite{peteCornerPercolationMathbb2008a} to have only finite clusters, with irrational critical exponents. The biased case got a cursory mention in~\cite{peteCornerPercolationMathbb2008a}, addressed in detail in~\cite{cornerPerco2025}, where it is proved that almost surely it has infinitely many infinite clusters.

\begin{theorem}
    Biased corner percolation on $\Z^2$ has indistinguishable infinite clusters.
\end{theorem}

See \Cref{thm: corner percolation has indistinguishable clusters} for the precise version and the proof.

For the infinite cycles $(\pi^n(x))_{n\in\Z}$ of an invariant ergodic permutation on a left Cayley graph of a group $\Gamma$, an immediate consequence of indistinguishability is that the stationary sequence of the jumps $\left(\pi^{n}(o)^{-1}\pi^{n+1}(o)\right)_{n\in\Z}$, conditioned on the cycle of $o$ to be infinite, is ergodic. This automatically applies to the interchange process; for the loop $O(n)$ model, one can a choose a uniform random orientation of the clusters; for corner percolation, as we will see, there is a natural measurable way to choose an orientation in each infinite cluster. The ergodicity in this last case answers Conjecture 5.3 in \cite{cornerPerco2022}, a previous version of \cite{cornerPerco2025}. See the end of Subsection~\ref{s.corner} for the details.

\begin{figure}[hbtp]
    \centering
    \begin{subfigure}[t]{0.45\textwidth}
            \includegraphics[width=\textwidth]{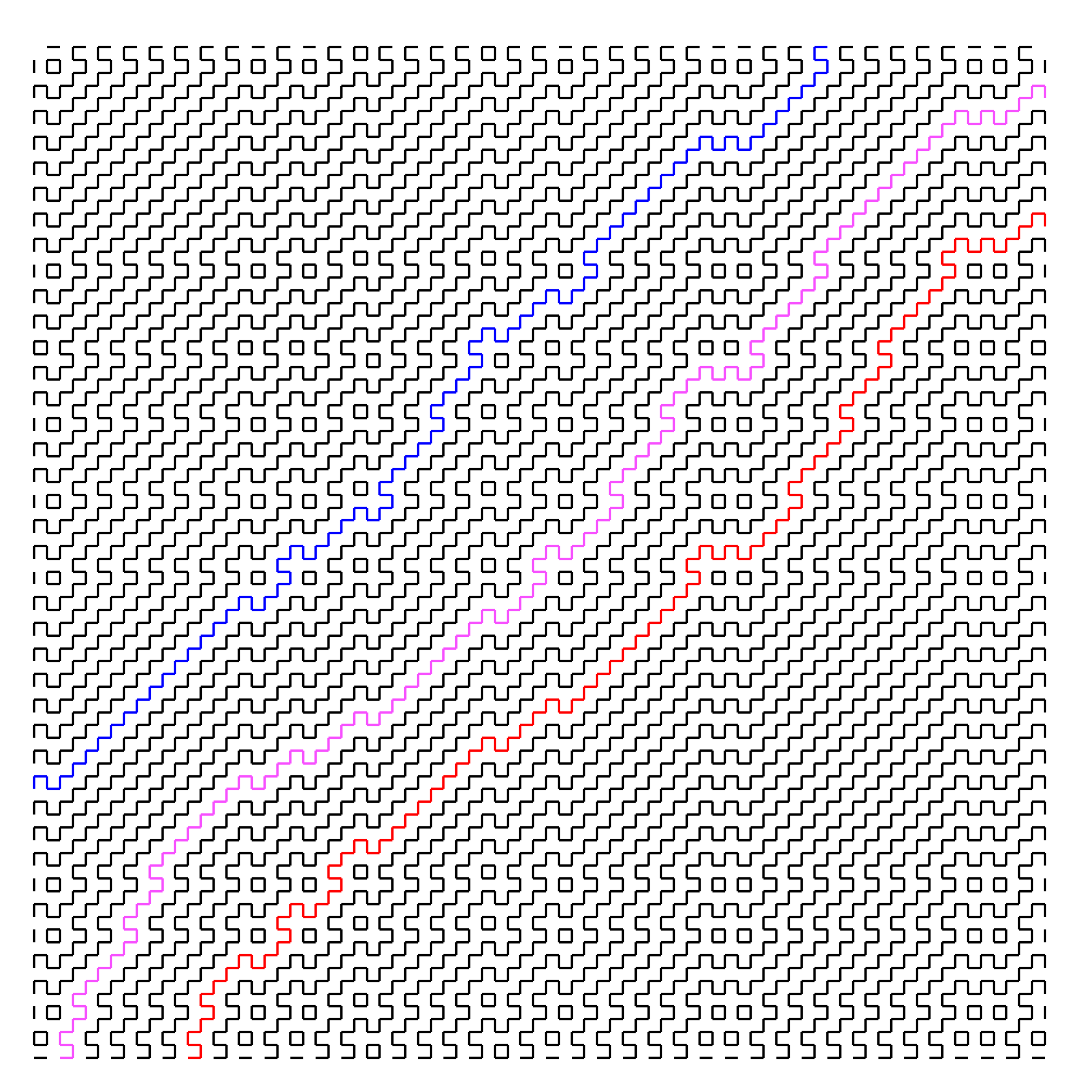}
%\caption{Corner percolation with $(p,q)=(0.2,0.8)$. It contains almost surely infinite paths that have an asymptotic slope of 1.}
    \end{subfigure}
    \hfill
    \begin{subfigure}[t]{0.45\textwidth}
            \includegraphics[width=\textwidth]{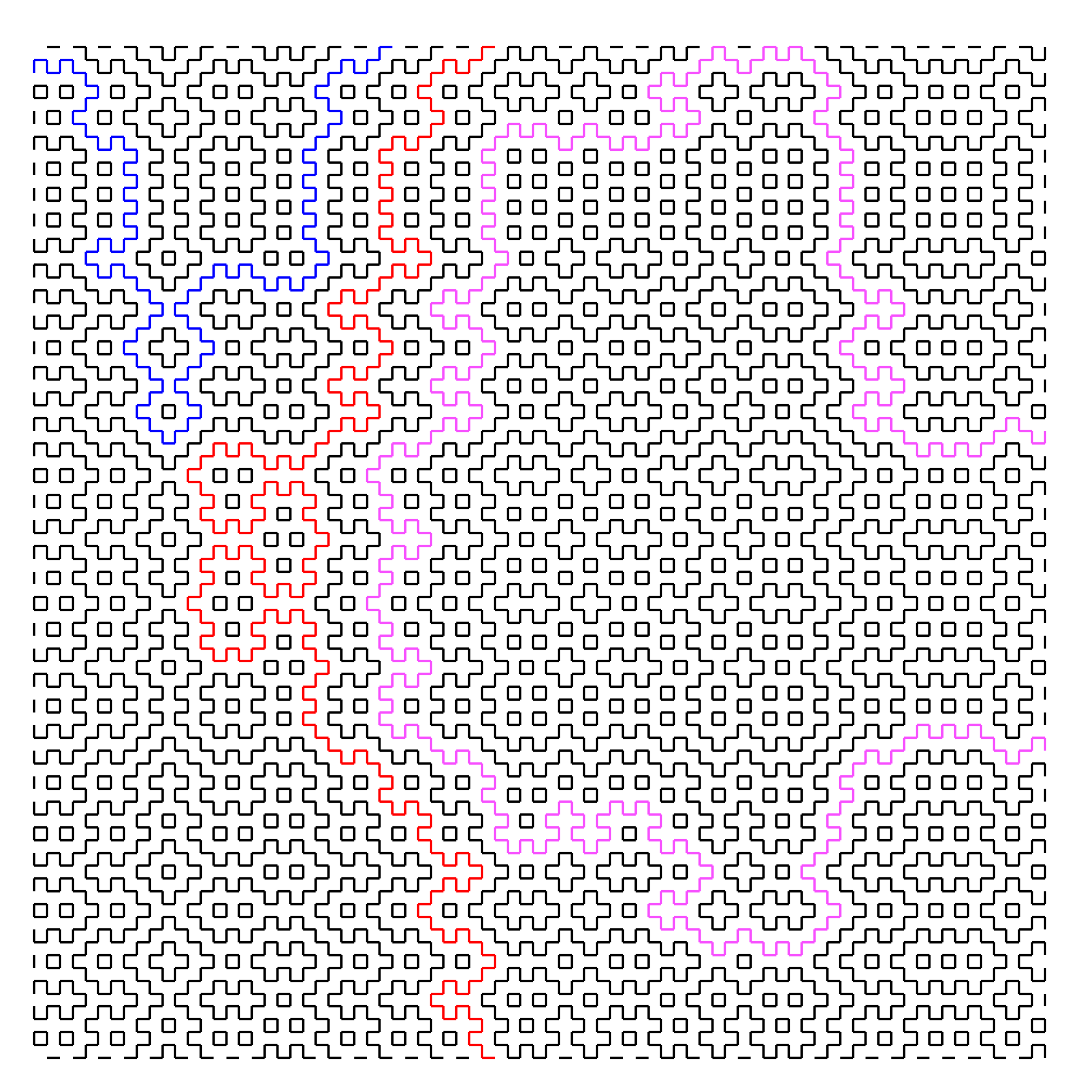}
            %\caption{Corner percolation with $(p,q)=(0.5,0.5)$. The asymptotic slope seen in the previous figure is not present for these values of $(p,q)$, even on a larger scale, and there are almost surely only finite paths.}
    \end{subfigure}
\caption{On the left, corner percolation with $(p,q)=(0.2,0.8)$. It contains almost surely infinite paths that have an asymptotic slope of 1. On the right, $(p,q)=(0.5,0.5)$. In that case there are almost surely only finite paths.}
\end{figure}

%In \cite{lyonsIndistinguishabilityPercolationClusters1999}, the authors use random walks on the clusters of insertion-tolerant percolation to show indistinguishability of the infinite clusters.
%Our \Cref{cor: technical on graphs} allows us to show indistinguishability in a broader range of models with $\infty$-ended clusters that are not exactly insertion tolerant.
%We could, as did the authors of \cite{lyonsIndistinguishabilityPercolationClusters1999}, use Bernoulli percolation as our example.
%We will instead use a model on which their result does not formally apply: the {\bf Poisson zoo}.

Our \Cref{cor: technical on graphs} also provides a slightly different proof of the Lyons-Schramm \indy\ theorem \cite{lyonsIndistinguishabilityPercolationClusters1999}, more robust, applicable to a broader range of models with $\infty$-ended clusters, which are not exactly insertion tolerant. As an example, we will treat the {\bf Poisson zoo}, to which their result does not formally apply.

In this site percolation model \cite{ráth2022percolationworms}, we consider a random finite nonempty subset of vertices inducing a connected subgraph and containing the origin $o$ (a \textbf{lattice animal}), and place an independent Poisson number with parameter $\lambda>0$ of independent copies of this lattice animal around each vertex of a Cayley graph.
When the animal in question is almost surely a single vertex, this is just Bernoulli site percolation. It is easy to see if the expected cardinality of a lattice animal is finite, then the model is non-trivial in the sense that not all of $G$ is covered, but we get infinite clusters for $\lambda>0$ large enough whenever $p_c<1$ for Bernoulli percolation --- which is all non-1-dimensional Cayley graphs \cite{duminil2020existence, easo2024counting}. The main results of \cite{ráth2022percolationworms} and \cite{PeteRokob} give many interesting examples where already an arbitrarily small density $\lambda>0$ is enough for percolation. Now, this model is not insertion tolerant in general, as one can always insert some finite shape but not necessarily a single vertex. This is however enough to show that a Poisson zoo on any amenable Cayley graph $G$ has almost surely at most one infinite cluster, using the classic argument from \cite{burtonkeane}. On the other hand, it is shown in \cite{bowenFinitaryRandomInterlacements2019} that, for any nonamenable Cayley graph, there exists $\lambda$ and some animal distribution (specifically taking value in worms, i.e., finite simple paths) such that there are infinitely many infinite clusters.

\begin{theorem}
    \label{thm: indis for poisson zoo}
    Any Poisson zoo on any nonamenable Cayley graph $G$ has indistinguishable infinite clusters.
\end{theorem}

In \cite[Proposition 1.4]{PeteRokob}, indistinguishability (for a truly insertion tolerant model) was used in proving the existence of a {\it unique} infinite cluster in a Poisson zoo process on $\mathbb{T}_3\times\Z^5$ at arbitrarily small $\lambda>0$ densities. 
   
   \begin{figure}[hbtp]
     \centering
    \begin{subfigure}[t]{0.45\textwidth}
    \includegraphics[width=\textwidth]{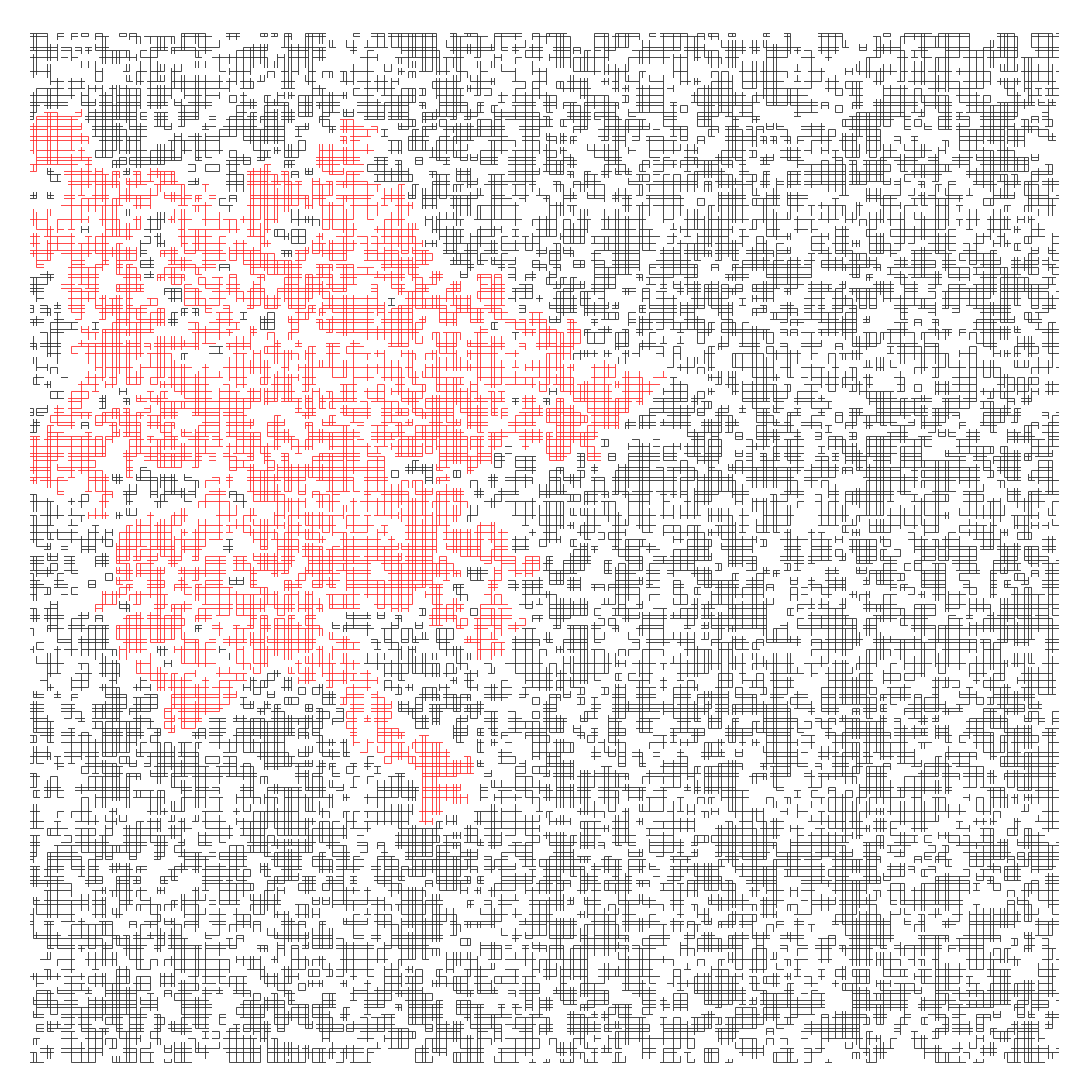}
    %\caption{Poisson Zoo where the animals are boxes of exponential size.}
    \end{subfigure}
    \hfill
    \begin{subfigure}[t]{0.45\textwidth}
        \includegraphics[width=\textwidth]{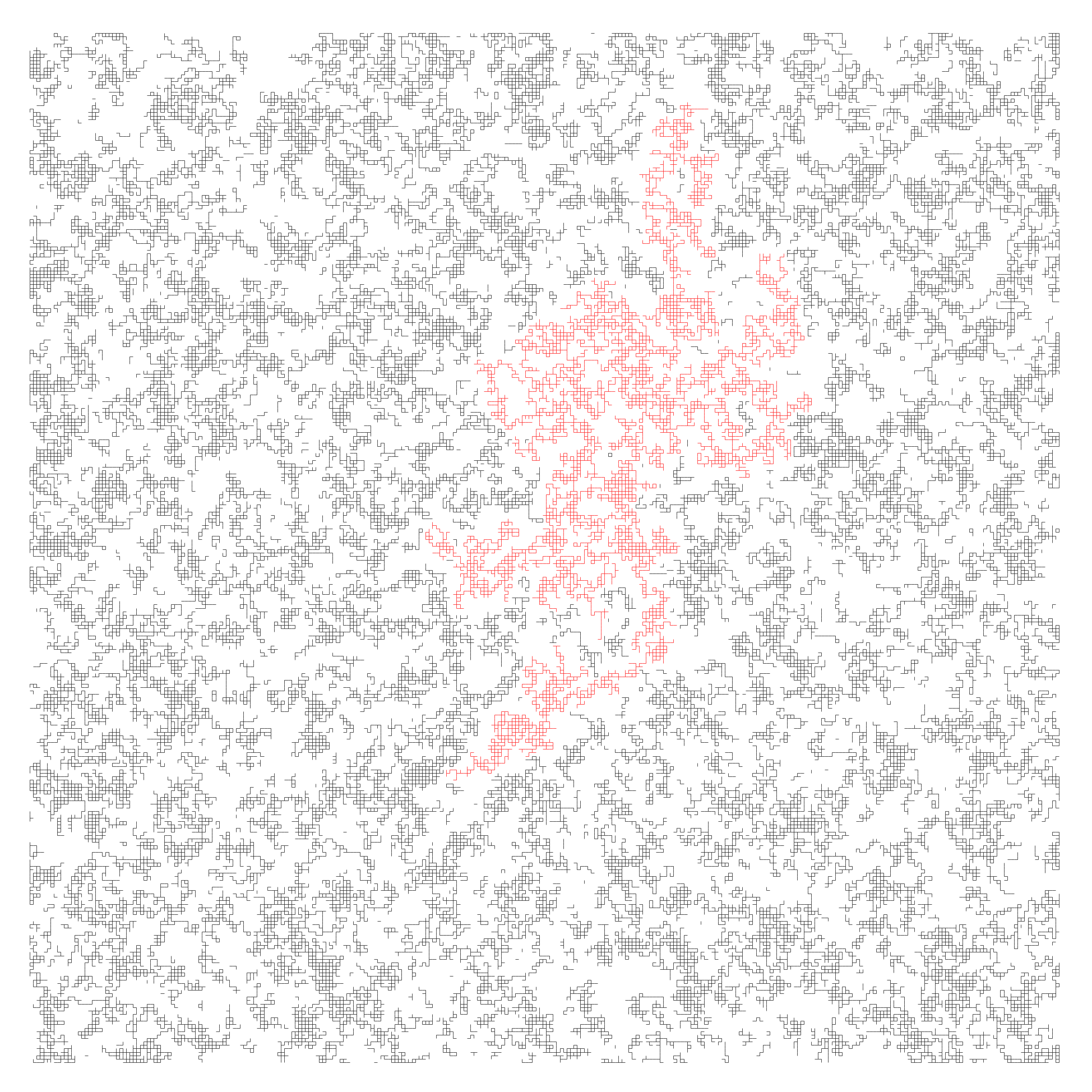}
            %\caption{Poisson Zoo where the animals are random walks of exponential length.}
    \end{subfigure}
\caption{On the left, a Poisson Zoo where the animals are boxes of exponentially distributed size. On the right, the animals are random walks of exponentially distributed length.}
\end{figure}

Our grafting method is much less useful in the context of one-ended infinite clusters. Tellingly, one of the main indistinguishability results that does not fit into our framework is the case of the Wired Uniform Spanning Forest in \cite{hutchcroft2017indistinguishability}, where all the components are one-ended trees. The problem is that grafting here is very destructive, as it involves replacing every vertex but finitely many in a cluster of a fixed vertex. As such, only properties that can be observed in a local manner, called \textbf{not essentially tail} (see Subsection~\ref{ss:nontail} for the precise definition), can be preserved after grafting. This still allows us to show the following partial indistinguishability result, extending a result from the Hutchroft and Nachmias article \cite{hutchcroft2017indistinguishability}, where specific properties of the WUSF were used: 

\begin{theorem}
\label{thm: nontail indis}
Let $\eta$ be a $\Gamma$-invariant, ergodic labelling of $G$, and $\omega$ an $\eta$-measurable site or bond percolation on $G$ such that $(\eta,\omega)$ is jointly $\Gamma$-invariant. There are no nontrivial not essentially tail cluster properties for $\omega$.
\end{theorem}

There are numerous open problems. A further model where two-ended clusters arise naturally and our methods might be applicable is the {\bf annealed random spatial permutation model} in $\R^d$ from \cite{betz2009spatial}. A mysterious model with long range correlations and limited surgery tolerance, with possibly multiple infinite clusters, is the {\bf Bernoulli line percolation} model \cite{HILARIO20195037} on $\Z^d$. For the {\bf Wired Minimal Spanning Forest}, even \indy\ of trees by the number of ends is open on amenable Cayley graphs where $\theta(p_c)=0$ is not known \cite[Chapter 11]{LPbook}. In the nonamenable setting, an interesting family of models are the invariant percolations arising from {\bf spaces with measured walls} \cite{MukhReck}. And how about the case of the {\bf loop $O(n)$ model} on nonamenable Cayley graphs, left open by our \Cref{thm: O(n) indistinguishability}, or the loop models associated to the planar site percolation models addressed in \cite{GlazmanHZ}?
\medskip

In the rest of this paper, in \Cref{sec:defs} we define the main concepts and provide the main lemmas, then we prove indistinguishability in each different setting in \Cref{sec:proofs}.

\section*{Acknowledgements}

We thank Miklós Abért for asking about the indistinguishability of infinite cycles in the interchange process. 

Our work was supported by the Hungarian National Research, Development and Innovation Office, Advanced grant 151155 (El Alami) and grant K143468 (Pete), by the ERC grant 810115-DYNASNET (Pete and Timár), and by the Icelandic Research Fund grant 239736-051 (Timár).

\section{Definitions and general lemmas \label{sec:defs}}

\subsection{Definitions}\label{ss:defs}

A \textbf{network} is a locally finite countable graph $G$ together with a measurable labelling $\eta\in L^G$ of its vertices, where $L$ is a chosen Polish space.
We could also take a labelling of the edges, or both edges and vertices with minimal modifications. 
We will in the future fix $L$ and denote a network by $(G,\eta)$.

A \textbf{network homomorphism} $\varphi$ from one network $(G,\eta)$ to another $(G',\eta')$ is a homomorphism from $G$ to $G'$ such that $\eta(v)=\eta'(\varphi(v))$ for any $v\in G$.
If bijective, it is called a \textbf{network isomorphism}.

A \textbf{rooted network} is a triplet $(G,\eta,o)$ where $o$ is a vertex of a \textit{connected} network $(G,\eta)$, and we call a network isomorphism from $(G,\eta)$ to $(G',\eta')$ that sends $o$ onto $o'$ a \textbf{rooted network isomorphism}. 
Let $\mathcal G_*^L$ be the set of rooted, networks up to rooted network isomorphisms.
We endow this set with the topology in which two rooted networks are close if they are isomorphic in a large ball around their roots $o$ and $o'$, and each label in the ball around $o$ is close to the label of its image in the topology of $L$.

From now on, let $(G,\eta,o)$ be a random rooted network with distribution $\mu$, and $\Pcal_{G,\eta}$ a $(G,\eta,o)$-measurable partition of $G$ that is invariant under rerooting.
We write $\Pcal_{G,\eta}(o)$ for the unique set of $\Pcal_{G,\eta}$ containing $o$.

We call any measurable subset $\A$ of $\mathcal G_{*}^L$ that is invariant under rerooting in $\Pcal_{G,\eta}(o)$ a \textbf{component property}.
We say that $\Pcal_{G,\eta}$ has \textbf{indistinguishable infinite components} if, for every component property $\A$, the event $\{(G,\eta,o)\in\A\}$ has either probability 0 or 1 conditionally on the component of $o$ being infinite.

\begin{lemma}
\label{lem:technical lemma for indis}
    Let $\A$ be a component property and let $(G,\eta,o)$ be a random rooted network.
    Let $(v_n)$ be a stationary sequence started at $v_0=o$ on $(G,\eta,o)$, meaning that $(G,\eta,v_n)$ has the same distribution as $(G,\eta,o)$, and such that $v_n\in\Pcal_{G,\eta}(o)$ for every $n\in\N$. It can depend on $(G,\eta,o)$ and additional randomness.
    Furthermore, let $S_{G,\eta,o}$ be a $(G,\eta,o)$-measurable subset of $G$ that is ``small'' for $(v_n)$, meaning that almost surely for every $R>0$, $(v_n)$ spends finitely many steps in the $R$-neighbourhood of $S_{G,\eta,o}$.

    Suppose that there exists $(G',\eta',o')$, $(v'_n)$ and $S'_{G',\eta',o'}$ %such that $((G,\eta,o),(v_n),S_{G,\eta,o})$ and $((G',\eta',o'),(v'_n),S_{G',\eta',o'})$ have the same distribution,
     with all the same properties assumed in the previous paragraph for $(G,\eta,o)$, $(v_n)$ and $S_{G,\eta,o}$,
    and that they are coupled in such a way that there exists $M,N\in\N$ such that on some event $\W$, for every $n\in\N$, the component $K$ of $v_{n+N}$ in $G\setminus S_{G,\eta,o}$ and the component $K'$ of $v'_{n+M}$ in $G'\setminus S'_{G',\eta',o'}$ are rooted network isomorphic with roots $v_{n+N}$ and $v'_{n+M}$ and environments $\eta$ and $\eta'$.
    
    Then conditionally on $\W$, up to probability zero events, $(G,\eta,o)\in\A$ if and only if $(G',\eta',o')\in\A$.
\end{lemma}

Typical stationary sequences $(G,\eta,v_n)$ when $G$ is a fixed Cayley graph and $\eta$ is an invariant percolation are the Simple Delayed Random Walk on the percolation cluster and the cycle of $o$ in a random permutation, both discussed in \Cref{lem: random permutation and SDRW are stationary}, as well as a uniform point in a hyperfinite exhaustion of the cluster of $o$, see \cite{hutchcroft2017indistinguishability}.

For an example of smallness of an infinite set w.r.t.~a stationary process, take simple random walk on $\Z^4$ decorated by i.i.d.~variables on the vertices, and $S$ the $x$-axis.

We now give a more immediately applicable version in the context of percolation with an environment on Cayley graphs. In some cases, such as \Cref{s.corner}, we need to apply \Cref{lem:technical lemma for indis}, but this special case might be more familiar.
Here, $G$ is a (left) Cayley graph of a finitely generated group $\Gamma$; every $g\in\Gamma$ acts on $G$ by $g\cdot x = xg^{-1}$, translating $g$ to the identity of $\Gamma$, which we call $o$.

\begin{proposition}
\label{cor: technical on graphs}
       Let $G$ be a Cayley graph of a finitely generated group $\Gamma$, $\eta$ an invariant labelling of $G$ and $\omega$ an $\eta$-measurable percolation process such that $(\eta,\omega)$ is jointly $\Gamma$-invariant.

       Let $(\eta',\omega')$ be an identically distributed copy of $(\eta,\omega)$, $x$ and $x'$ two vertices, $E$ an $(\eta,\eta')$-measurable event on which the clusters $K$ of $x$ in $\omega$ and $K'$ of $x'$ in $\omega'$ are infinite, and assume that, conditionally on $E$, $\eta$ and $\eta'$ only differ in a finite set $S$ such that $K\setminus S$ and $K'\setminus S$ share at least one infinite connected component $K_0$.

    Let $\A$ be a component property. Then, conditionally on $E$, $x$ has type $\A$ in $\eta$ if and only if $x'$ has type $\A$ in $\eta'$ (up to a probability 0 event).
\end{proposition}

\begin{figure}[htbp]
 \centering
    \begin{subfigure}[t]{0.3\textwidth}
    \def\svgwidth{\textwidth}
    %% Creator: Inkscape 1.4.2 (ebf0e940d0, 2025-05-08), www.inkscape.org
%% Accompanies image file 'lemma2_4_1.pdf' (pdf, eps, ps)
\begingroup%
  \makeatletter%
  \providecommand\color[2][]{%
    \errmessage{(Inkscape) Color is used for the text in Inkscape, but the package 'color.sty' is not loaded}%
    \renewcommand\color[2][]{}%
  }%
  \providecommand\transparent[1]{%
    \errmessage{(Inkscape) Transparency is used (non-zero) for the text in Inkscape, but the package 'transparent.sty' is not loaded}%
    \renewcommand\transparent[1]{}%
  }%
  \providecommand\rotatebox[2]{#2}%
  \newcommand*\fsize{\dimexpr\f@size pt\relax}%
  \newcommand*\lineheight[1]{\fontsize{\fsize}{#1\fsize}\selectfont}%
  \ifx\svgwidth\undefined%
    \setlength{\unitlength}{1133.85826772bp}%
    \ifx\svgscale\undefined%
      \relax%
    \else%
      \setlength{\unitlength}{\unitlength * \real{\svgscale}}%
    \fi%
  \else%
    \setlength{\unitlength}{\svgwidth}%
  \fi%
  \global\let\svgwidth\undefined%
  \global\let\svgscale\undefined%
  \makeatother%
  \begin{picture}(1,1)%
    \lineheight{1}%
    \setlength\tabcolsep{0pt}%
    \put(0,0){\includegraphics[width=\unitlength,page=1]{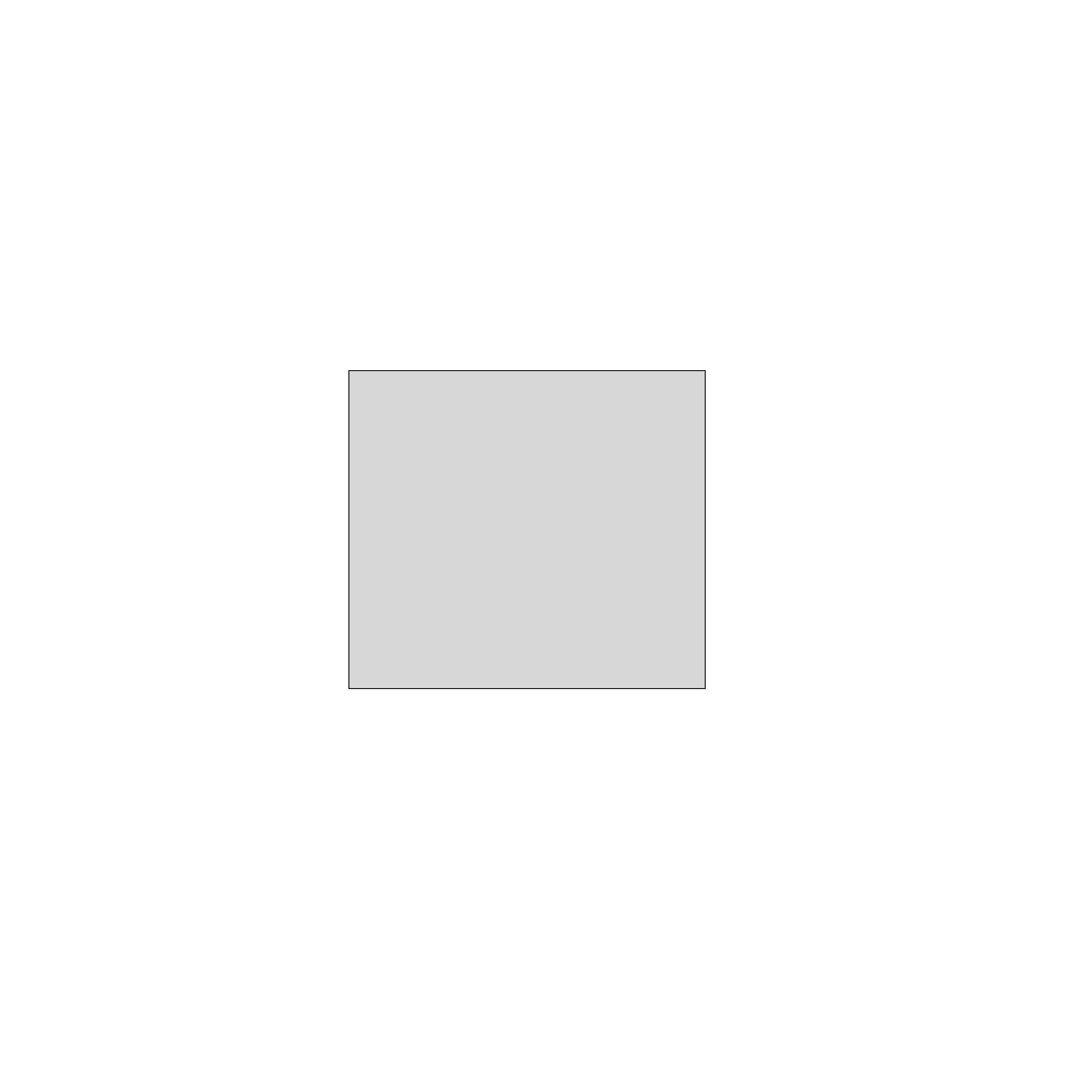}}%
    \put(0.3306678,0.38165846){\makebox(0,0)[lt]{\lineheight{1.25}\smash{\begin{tabular}[t]{l}$S$\end{tabular}}}}%
    \put(0,0){\includegraphics[width=\unitlength,page=2]{Images/lemma2_4_1.pdf}}%
    \put(0.0633739,0.79534941){\makebox(0,0)[lt]{\lineheight{1.25}\smash{\begin{tabular}[t]{l}$\color{red}K_0$\end{tabular}}}}%
    \put(0.80541435,0.92247348){\makebox(0,0)[lt]{\lineheight{1.25}\smash{\begin{tabular}[t]{l}$K$\end{tabular}}}}%
  \end{picture}%
\endgroup%

  %  \caption{The percolation obtained from $\eta$ with the cluster $K$ and its branch $K_0$.}
    \end{subfigure}
    \hglue 2 cm
%\hfill
    \begin{subfigure}[t]{0.3\textwidth}
    \def\svgwidth{\textwidth}
    %% Creator: Inkscape 1.4.2 (ebf0e940d0, 2025-05-08), www.inkscape.org
%% Accompanies image file 'lemma2_4_2.pdf' (pdf, eps, ps)
\begingroup%
  \makeatletter%
  \providecommand\color[2][]{%
    \errmessage{(Inkscape) Color is used for the text in Inkscape, but the package 'color.sty' is not loaded}%
    \renewcommand\color[2][]{}%
  }%
  \providecommand\transparent[1]{%
    \errmessage{(Inkscape) Transparency is used (non-zero) for the text in Inkscape, but the package 'transparent.sty' is not loaded}%
    \renewcommand\transparent[1]{}%
  }%
  \providecommand\rotatebox[2]{#2}%
  \newcommand*\fsize{\dimexpr\f@size pt\relax}%
  \newcommand*\lineheight[1]{\fontsize{\fsize}{#1\fsize}\selectfont}%
  \ifx\svgwidth\undefined%
    \setlength{\unitlength}{1133.85826772bp}%
    \ifx\svgscale\undefined%
      \relax%
    \else%
      \setlength{\unitlength}{\unitlength * \real{\svgscale}}%
    \fi%
  \else%
    \setlength{\unitlength}{\svgwidth}%
  \fi%
  \global\let\svgwidth\undefined%
  \global\let\svgscale\undefined%
  \makeatother%
  \begin{picture}(1,1)%
    \lineheight{1}%
    \setlength\tabcolsep{0pt}%
    \put(0,0){\includegraphics[width=\unitlength,page=1]{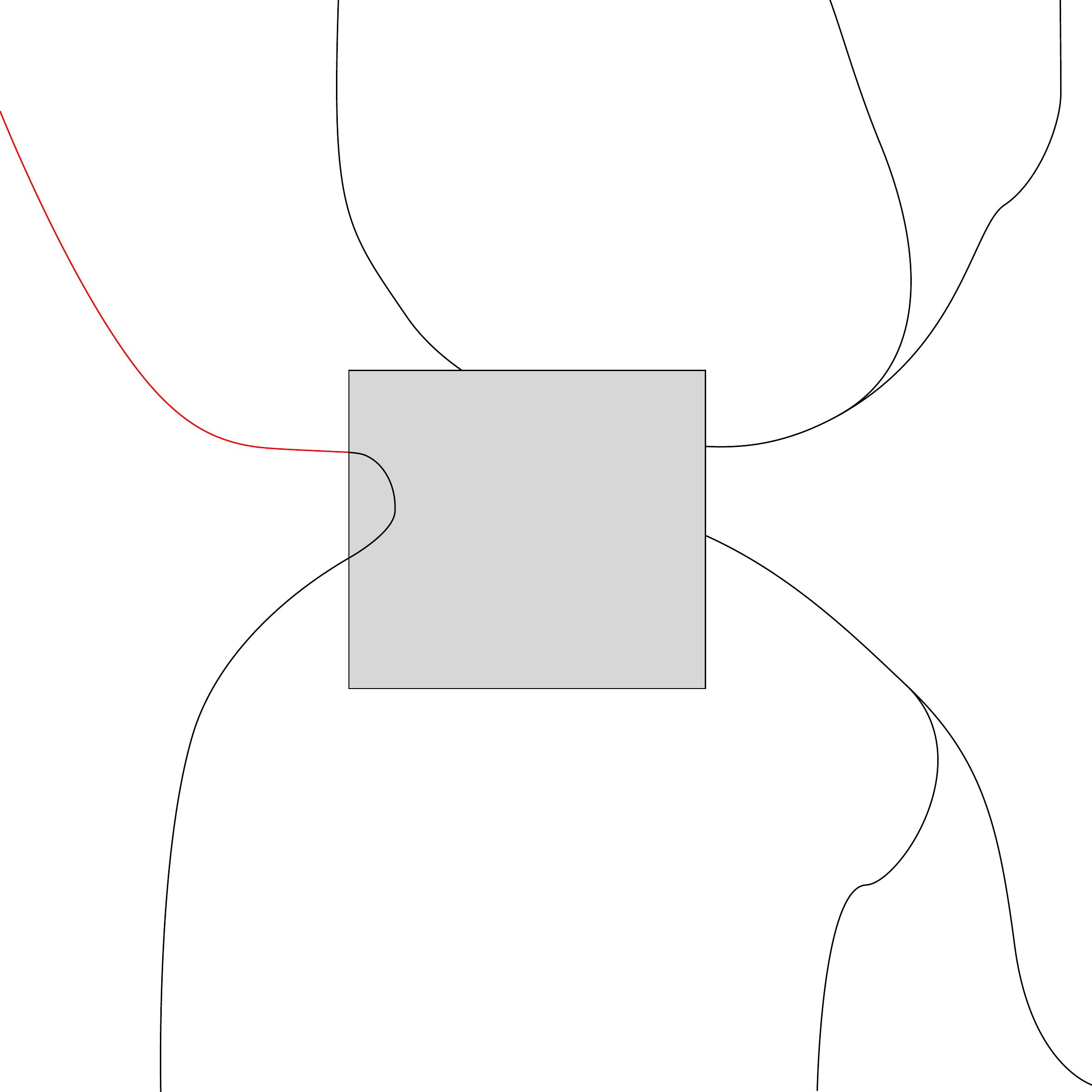}}%
    \put(0.0633739,0.79534941){\makebox(0,0)[lt]{\lineheight{1.25}\smash{\begin{tabular}[t]{l}$\color{red}K_0$\end{tabular}}}}%
    \put(0.3306678,0.38165846){\makebox(0,0)[lt]{\lineheight{1.25}\smash{\begin{tabular}[t]{l}$S$\end{tabular}}}}%
    \put(0.16194427,0.11102928){\makebox(0,0)[lt]{\lineheight{1.25}\smash{\begin{tabular}[t]{l}$K'$\end{tabular}}}}%
    \put(0,0){\includegraphics[width=\unitlength,page=2]{Images/lemma2_4_2.pdf}}%
  \end{picture}%
\endgroup%

%    \caption{The percolation obtained from $\eta'$ with the cluster $K'$ and its branch $K_0$.}
    \end{subfigure}
%    \caption{$\eta$ and $\eta'$ only differ in $F$, so for a given component property $\A$, either both $K$ and $K'$ have type $\A$ or neither do because they share the infinite branch $K_0$.}
\caption{In the setup of \Cref{cor: technical on graphs}, the configurations $\eta$ and $\eta'$ only differ in $S$. In the first, $K_0$ is part of the cluster $K$, in the second, it is part of $K'$.}
\end{figure}

\begin{remark}
\label{rem: weaker condition smallness instead of finiteness}
In this lemma, we ask for the difference between $\eta$ and $\eta'$ to be localized to a finite set. We could instead have the weaker requirement that, conditionally on $E$, for any $R\in\N$, there are only finitely many $x\in K_0$ such that $x$ is in the $R$-neighbourhood of $S$. This condition mirrors the smallness of $S$ in \Cref{lem:technical lemma for indis} for a random walk escaping to infinity in $K_0$.
The stronger assumption is easier to read; the weaker one will be needed in the case of models like corner percolation where it is impossible to make any finite modification.
\end{remark}

A very powerful tool on determistic (and random) graphs that we will need in the following sections is the \textbf{mass transport principle}, introduced to percolation theory in \cite{haggstromInfiniteClustersDependent1997}, developed further in \cite{benjamini1999group}.

\begin{lemma}[Mass transport principle]
Let $G$ be a Cayley graph of a finitely generated group $\Gamma$ and 
    let $\phi:G\times G\lora \R^+$ be diagonally invariant under the action of $\Gamma$, meaning that for any $g\in\Gamma$ and $x,y\in G$, we have $\phi(x,y)=\phi(g\cdot x,g\cdot y)$.
    Then $$\sum_{x\in G}\phi(x,o)=\sum_{y\in G}\phi(o,y).$$
\end{lemma}
This property is true on a larger class of graphs called \textit{unimodular transitive graphs}, and on so called \textit{unimodular random graphs and networks} (see \cite{aldousProcessesUnimodularRandom2007}).

\subsection{Proofs of the key general tools}

We will prove the following more general version of \Cref{lem:technical lemma for indis}:
\begin{lemma}\label{lem:technical extended version}
    Let $\A$ be a component property and let $(G,\eta,o)$ be a random rooted network.
        Let $(v_n)$ be a stationary sequence started at $v_0=o$ on $(G,\eta,o)$, meaning that $(G,\eta,v_n)$ has the same distribution as $(G,\eta,o)$, and such that $v_n\in\Pcal_{G,\eta}(o)$ for every $n\in\N$. It can depend on $(G,\eta,o)$ and additional randomness.
    Furthermore, let $S_{G,\eta,o}$ be a $(G,\eta,o)$-measurable subset of $G$ and $\psi$ a random increasing function $\N\longrightarrow \N$ such that, almost surely:
    \begin{itemize}
        \item $\limsup_{n\rightarrow \infty}\frac n{\psi(n)} >0$,
        \item $S$ is ``smallish'' for $(v_{\psi(n)})$, in the following sense, a relaxed version of being ``small'' in \Cref{lem:technical lemma for indis}: for every $R\in\N$, $\limsup_{N\to\infty} \big|\big\{1\leq n\leq N : B(v_{\psi(n)},R) \cap S \not= \emptyset \big\}\big|/N = 0$.
    \end{itemize}

        Suppose there are $(G',\eta',o')$, $(v'_n)$, $S'_{G',\eta',o'}$ and $\psi':\N\longrightarrow\N$ %such that $((G,\eta,o),(v_n),S_{G,\eta,o})$ and $((G',\eta',o'),(v'_n),S_{G',\eta',o'})$ have the same distribution, 
         with all the same properties assumed in the previous paragraph for $(G,\eta,o)$, $(v_n)$, $S_{G,\eta,o}$ and $\psi$,
        and that they are coupled in such a way that on some event $\W$, for every $n\in\N$, the component $K$ of $v_{\psi(n)}$ in $G\setminus S_{G,\eta,o}$ and the component $K'$ of $v'_{\psi'(n)}$ in $G'\setminus S'_{G',\eta',o'}$ are isomorphic rooted networks with roots $v_{\psi(n)}$ and $v'_{\psi'(n)}$ and environments $\eta$ and $\eta'$.
    
    Then, conditionally on $\W$, up to probability zero events, $(G,\eta,o)\in\A$ if and only if $(G',\eta',o')\in\A$.
\end{lemma}

A useful aspect of the lemma is that one does not need to couple the full sequence $(v_n)$ with $(v_n')$, but only the subsequence $(v_{\psi(n)})$ with $(v'_{\psi'(n)})$.

 For simple random walk on a connected locally finite infinite graph, any finite set $S$ is smallish (because its stationary measure is finite, while the total measure is infinite). For an example of a smallish infinite set, take simple random walk on $\Z^3$, and $S$ the $x$-axis.

\begin{proof}[Proof of \Cref{lem:technical extended version}]
By the measurability of component properties, we can for any fixed $\e>0$ find a subset $A_\e$ of $\mathcal G^L_{*}$ that approximates $\A$ with error $\e$, in the sense that $\mu(\A\bigtriangleup A_\e)$< $\e$, 
where $\{(G,\eta,o)\in A_\e\}$ only depends on a finite ball of radius $R_\e$ around $o$ and the values of $\eta$ in this ball.

It immediately follows that $A_\e^c$ approximates with error $\e$ the property $\lnot \A$.

We define the {\bf density of $A_\e$ along $(v_n)$} in $(G,\eta,o)$ as
$$
D_{G,\eta,o}(A_\e,(v_n)):=\lim_{n\rightarrow\infty}\frac 1 n \sum_{i=0}^{n-1}\mathbb 1((G,\eta,v_n)\in A_\e).
$$

This limit exists (though might be random), thanks to Birkhoff's theorem.
Conditionally on $o=v_0$ having type $\A$, $(v_n)$ is still stationary because for any $n\in\N$, and any measurable $E$ in $\mathcal G^L_{*}$,
$\P( (G,\eta,o)\in E,\  (G,\eta,o)\in\A)=\P((G,\eta,v_n)\in E,\  (G,\eta,v_n)\in\A)$ by stationarity.
Applying Birkhoff's theorem again,
$$
\E\big(D_{G,\eta,o}(A_\e,(v_n))\mid (G,\eta,o)\in\A\big)=\P\big((G,\eta,o)\in A_\e\mid (G,\eta,o)\in\A\big)>1-\e.
$$
Using Markov's inequality, for any $k$ in $\N_+$,
\begin{equation}
\label{eq: Markov applied to density}
\P\big(D_{G,\eta,o}(A_\e,(v_n))>1-1/k\mid (\eta,v_0)\in\A\big)>1-k\e.
\end{equation}

Similarly, 
$$\E\big(D_{G,\eta,o}(A^c_\e,(v_n)) \mid (G,\eta,o)\notin\A\big)=\P((G,\eta,o)\in A^c_\e\mid (G,\eta,o)\notin\A)>1-\e.$$
Using again Markov's inequality,
$$
\P\big(D_{G,\eta,o}(A_\e,(v_n))<1/k\mid (G,\eta,o)\notin\A\big)>1-k\e.
$$

This means that, conditionally on $o$ having type $\A$, using Borel-Cantelli's lemma, there are almost surely only finitely many $n\in\N$ such that $D_{G,\eta,o}(A_{2^{-n}},(v_m))<1-1/k$.
Symmetrically, conditionally on $o$ having type $\lnot\A$, there are almost surely only finitely many $n\in\N$ such that $D_{G,\eta,o}(A_{2^{-n}},(v_m))>1/k$.
Since we have this for every $k\in\N_+$, we get that $(D_{G,\eta,o}(A_{2^{-n}},(v_m)))_n$ converges a.s.\ to 0 or 1, depending on the type of $o$ in $(G,\eta)$.

On the other hand, there exists by assumption a $\delta=\delta\big((G,\eta,o),(v_n),(G',\eta',o'),(v'_n)\big)>0$ such that we can find $l,l'$ arbitrarily large such that the proportion of elements of $(v_m)^l_{m=0}$ and $(v'_m)_{m=0}^{l'}$ that are in $(v_{\psi(m)})$ and $(v'_{\psi'(m)})$, respectively, is at least $\delta$.
The smallishness of $S$ implies that, for any $R$, the density of values $n\in\N$ such that
$v_{\psi(n)}$ is in an $R$-neighbourhood of $S_{G,\eta,o}$ is zero, and similarly for $v'_n$ and $S'_{G,\eta',o'}$. 
Hence the densities satisfy $|D_{G,\eta,o}(A_\e,(v_m))-D_{G',\eta',o'}(A_\e,(v'_m))|\leq 1-\delta$ almost surely on $\W$, for any $\e>0$ fixed. As both sequences $(D_{G,\eta,o}(A_{2^{-n}},(v_m)))_n$ and $D_{G',\eta',o'}(A_{2^{-n}},(v'_m)))_n$ converge to either 1 or 0, their limits are almost surely the same. Using the previous paragraph, this implies that $(G,\eta,o)$ and $(G',\eta',o')$ have almost surely the same $\A$-type conditionally on $\W$.
\end{proof}

To prove \Cref{cor: technical on graphs}, we will need two different types of stationary sequences in \Cref{lem:technical lemma for indis}; their stationarity is a special case of \cite[Theorem 4.1]{aldousProcessesUnimodularRandom2007}, but we give here a self-contained proof of both.

In the case of a Cayley graph $G$, a process $(v_n)_{n\in\N} \in G^\N$ is stationary if and only if $\P((\eta,v_n)\in B)=\P((\eta,v_m)\in B)$ for any $n,m\in\N$ and any $\Gamma$-invariant, measurable $B\subset L^G\times G$.

\begin{definition}
    We call \textbf{Simple Delayed Random Walk} (or \textbf{SDRW}) on some percolation configuration $\omega$ the Markov chain with the following transition probabilities:
    $$
p(x,y):=\begin{cases}
    \frac{1}{\deg_G(x)},&\text{if $x$ and $y$ are neighbours in $\omega$,}\\
    1-\frac{\deg_\omega(x)}{\deg_G(x)},&\text{if $x=y$,}\\
    0 &\text{else.}
\end{cases}
    $$
\end{definition}

In other words, the SDRW at $x$ chooses a random $G$-neighbour $y$ of $x$ uniformly, then travels to it if there is a percolation edge between $x$ and $y$ and stays on $x$ otherwise.
It will clearly always stay in the cluster of its starting point.

\begin{lemma}
             \label{lem: random permutation and SDRW are stationary}
             For any $\Gamma$-invariant random permutation $\pi$ on $G$, $(\pi^n(o))_{n\in\N}$ is stationary.
             
             For any invariant percolation $\omega$, the SDRW $(v_n)_{n\in\N}$ on $\omega$ is stationary.
         \end{lemma}

Recall that $G$ is a Cayley graph of $\Gamma$, hence we identify $\Gamma$ and the vertices of $G$.
         
    \begin{proof}
    First, the permutation $\pi$.
       Let $A$ be any event.
       We define the mass transport function function
       \begin{equation}
           \phi(x,y):=\P(y\cdot A,\pi^{-1}(y)=x).
       \end{equation}
This function is invariant under the diagonal action of $\Gamma$, by the $\Gamma$-invariance of $\pi$; we can then use the mass transport principle and obtain the following:
       \begin{align}
           \P(A)
           &=\sum_{x\in G} \P(A,\pi^{-1}(o)=x)\\
           &=\sum_{x\in G} \P(o\cdot A, \pi^{-1}(o)=x)\\
       &=\sum_{x\in G}\phi(x,o)
       =\sum_{y\in G} \phi(o,y)\\
           &= \sum_{y\in G} \P(y\cdot A, \pi^{-1}(y)=o)\\
           &= \P(\pi(o)\cdot A).
       \end{align}

Now, for the SDRW,
       let $A$ be any event and assume without loss of generality that $v_0=o$. Let us define the random function
    $$
    \phi(x,y):=\P(y\cdot A, y\sim_\omega x),
    $$
    where $x\sim_\omega y$ if $x$ and $y$ are neighbours in $\omega$.
    It is invariant under the diagonal action of $\Gamma$ by the $\Gamma$-invariance of $\omega$; we can then use the mass-transport principle and obtain the following:
    \begin{align*}
        \P(A)&=\sum_{x\in G}\P(A,v_1=x)\\
             &=\P(A,v_1=o)+        \frac 1 {\deg(o)} \sum_{x\in G} \P( o\cdot A,o\sim_{\omega }x)\\
             &=\P(A,v_1=o)+\sum_{x\in G} \phi(x,o)=\P(A,v_1=o)+\sum_{y\in G} \phi(o,y)\\
             &=\P(A,v_1=o)+\frac 1 {\deg(o)}\sum_{y\in G} \P(y\cdot A, y\sim_\omega o)\\
             &=\sum_{y\in G} \P(y\cdot A, v_1=y)\\
            &=\P(v_1\cdot A).
\end{align*}

The stationarity then follows from the fact that a SDRW is a time-homogeneous Markov chain.
   \end{proof}

We will first prove that if the assumptions of \Cref{cor: technical on graphs} are satisfied, then $K$ and $K'$ have the same number of ends. Recall that an infinite cluster in an invariant percolation on a Cayley graph can only be 1-, 2- or $\infty$-ended \cite{benjamini1999group}.

     If $K$ is $\infty$-ended, so is $K_0$, and thus so is $K'$. Indeed, assume for contradiction that $K_0$ is $1$ or $2$-ended while $K$ is $\infty$-ended with positive probability.
     Then $K$ would have an isolated end ($K_0$) which is impossible by \cite[Proposition 8.33]{LPbook}.

    Then, $K$ is one-ended if and only if $K'$ is one-ended as well.
    Indeed, according to Proposition 2.1 from \cite{bowen2022perfectmatchingshyperfinitegraphings}, an infinite cluster has linear growth if and only if it is two-ended.
    This means that if, say, $K$ is two-ended and $K'$ one-ended, then $K_0$ is one-ended, has superlinear growth and is included in $K$, contradicting the fact that $K$ has linear growth.

\begin{proof}[Proof of \Cref{cor: technical on graphs}]
We prove here the proposition under the slightly more general condition mentioned in \Cref{rem: weaker condition smallness instead of finiteness}.
We will construct for each possible number of ends a sequence that allows us to use \Cref{lem:technical lemma for indis} or its extension \Cref{lem:technical extended version}.
\medskip

Because $K$ and $K'$ have the same number of ends almost surely (as we concluded before this proof), we can rewrite $E$ as the disjoint union of $E\cap \{\text{$K$ and $K'$ have $i$ ends}\}$ for $i$ in $\{1,2,\infty\}$. These three events satisfy the assumptions of \Cref{cor: technical on graphs}, and if the lemma holds on these three events, then it holds for $E$.
We can thus assume that the number of ends of $K$ and $K'$ is an almost sure constant conditionally on $E$.

If $K$ and $K'$ are $\infty$-ended, we can use a reasoning as in \cite{lyonsIndistinguishabilityPercolationClusters1999}. Since $K_0$ is $\infty$-ended, as seen above, it is transient. 
Write $\partial_{V,\omega}^{in}K_0$ for the vertices of $K_0$ that have at least one $\omega$-neighbour outside of $K_0$, and let $B:=\partial_{V,\omega}^{in}K_0 \cup \partial_{V,\omega'}^{in}K_0$, an almost surely finite set. Clearly, $K_0\setminus B$ has at least one transient infinite component, call it $K_1$. The point of defining this $K_1$ is that for every $y\in K_1$, the neighbours of $y$ are the same in $\omega$ and $\omega'$. Now fix a $y\in K_1$, and take an arbitrary deterministic finite path in $K$ from $x$ to $y$, and an arbitrary finite path $P'$ in $K'$ from $x'$ to $y$. The SDRW from $y$ stays in $K_1$ with positive probability, and the trajectory in $K$ and $K'$, respectively, conditioned on this event, have the same distribution. So, on the event $E$, define the following coupling of the SDRW processes $(v_n)$ and $(v'_n)$. Start running them independently; if the starting segments are $P$ and $P'$, respectively, which happens with a positive probability, continue the coupling as follows; otherwise, continue independently. After $P$, sample the continuation of $(v_n)$ from $y$; if it stays in $K_1$, then give the same trajectory to the rest of $(v'_n)$ after $P'$; otherwise, continue $(v'_n)$ independently. Let $\W$ be the intersection of $E$ with the event that the above coupling was successful, i.e., that $(v_n)$ and $(v'_n)$ eventually stay in $K_1$ and coincide. Then we apply \Cref{lem:technical lemma for indis}.

\medskip

If $K$ and $K'$ are 2-ended, then $K_0$ is either one or two-ended.
The strategy is the same in both cases: we will extract an invariantly chosen, locally defined bi-infinite path that goes ``from one end of $K$ to the other'', and send $(v_n)$ along this path in a random direction.

To be more precise, for every cluster $\C$, choose uniformly one end to be direction 1 and the other to be direction 2.
    Remember that for any $y$ in a two-ended cluster $\mathcal C$, there is a minimal integer $\kappa(y)$ such that $\mathcal C\setminus B(y,\kappa(y))$ has two infinite connected components.
    We write $\kappa(\mathcal C):=\min_{y\in\mathcal C}\kappa(y)$. 
    By unimodularity, there are almost surely infinitely many $y$ in $\C$ with $\kappa(y)=\kappa(\C)$, infinite in both directions.
    We gather these in a set we call $\mathcal B(\mathcal C)$, and assign independently to each $y$ in $\mathcal B(\mathcal C)$ an independent random variable $u(y)$ uniform in $[0,1]$.
    If we then write $$\mathcal B'(\mathcal C):=\{x\in\mathcal B(\mathcal C)\mid \forall y \in \mathcal B(\mathcal C)\cap B(x,\kappa(\mathcal C)),\ u(x)<u(y)\},$$
    this is an infinite set, and has the following property:
    for any $y\in\mathcal B'(\mathcal C)$, every other vertex $z$ in $\mathcal B'(\mathcal C)$ is  in one of the two infinite parts of $\C\setminus B(y,\kappa(\C))$.
    Write $y\preceq_\C z$ if $z$ is in the part that contains direction 2, and $z\preceq_\C y$ if the reverse is true. 
    It is easy to see that this order is isomorphic to the usual ordering of $\Z$.

    We can then define a random permutation $\pi$, jointly invariant with $\eta$, by taking $\pi(y)=y$ if $y$ is not in a two-ended cluster, while $\pi(y)=z$ with $z$ being the successor of $y$ for $\preceq_\C$ if $y$ is in a two-ended infinite cluster $\C$.

    Then define $v_n:=\pi^n(x)$, construct the same way $\pi'$ on $\eta'$ using the same random variables $u(y)$, and define $v'_n:=(\pi')^n(x)$.
    These are stationary, transient random walks on $K$ and $K'$, and because they are locally defined, if they both enter and never leave $K_0$ after some time, they are asymptotically equal.
Using again the assumption on $S$ in \Cref{rem: weaker condition smallness instead of finiteness}, we can thus use \Cref{lem:technical lemma for indis} to finish the proof for this case.
\medskip

The case in which $K$ and $K'$ are both one-ended is more involved. We will use the SDRW, as in the $\infty$-ended case, but, as it is not necessarily transient on $K_0$, we will have to use~\Cref{lem:technical extended version}.

First note that $D:=(K\setminus K_0)\cup (K'\setminus K_0)$ is finite, otherwise $K$ or $K'$ would have at least two ends by the assumption on $S$ in \Cref{rem: weaker condition smallness instead of finiteness}, with the infinite $K_0$. (Nevertheless, $S$ still could be infinite.) Now, changing finitely many vertices and edges in a graph cannot change whether its unique infinite component is recurrent or transient for SDRW (e.g., because recurrence is equivalent to having infinite effective resistance to infinity \cite[Theorem 2.11]{LPbook}), thus either $K_0,K,K'$ are all recurrent, or all are transient. In the transient case, the proof is identical to the $\infty$-ended case handled above, hence we will assume recurrence.

Next we will construct a finite subgraph $D_1$ that contains $D$ and such that 
\begin{enumerate}
\item $K|_{D_1}$ and $K'|_{D_1}$ are connected,
\item $K\setminus D_1=K'\setminus D_1\subset K_0$ is connected, and 
\item for every vertex in $K\setminus D_1$, its boundary vertices in $K$ are the same as in $K'$.
\end{enumerate}
We have observed that $D$ consists of finitely many components and they are all finite. Take a one-ended spanning tree $T_0$ in $K_0$, and 
%extend it to a one-ended spanning tree $T$ (respectively $T'$) of $K$ (respectively $K'$).
choose a finite subtree $t$ of $T_0$ whose complement in $T_0$ is connected and that contains every vertex of $K_0$ that is at distance at most 2 from some vertex of $D$ in $K\cup K'$.
Let $D_1$ be the graph induced by $t\cup D$ in $G$. Now, item 1 holds because $K|_{D_1}$ consists of $t$ and the components of $ K\setminus K_0$ together with edges connecting them to $t$ (and possibly some further edges on these vertices). Similarly for $K'|_{D_1}$. Item 2 holds because $T_0\setminus t$ is a connected spanning subgraph of $K\setminus D_1$. Finally, item 3 is true because $K\setminus D_1$ is not adjacent to $D$ (by the choice of $t$), hence all its external boundary vertices are in $K_0 \subseteq K\cap K'$.

%{\color{purple}Every time the walks exit $D_1$ through an edge $e$ (one endpoint in $D_1$ and one endpoint in $D_1^c$), we will sample a trajectory inside $K\setminus D_1=K'\setminus D_1$ according to SDRW on $K$ (or, equivalently in the present case, on $K'$) until entering $D_1$ again, and define both $(v_n)$ and $(v'_n)$ to have this trajectory at this segment. So, for every $e$ in the edge boundary of $D_1$ and for every $k\in\N$, the $k$th excursions of $(v_n)$ and $(v'_n)$ exiting $D_1$ along $e$ will be the same. For the parts of $(v_n)$ and $(v_n)$ inside $D_1$, we cannot apply the same method, because $K|_{D_1}$ and $K'|_{D_1}$ may not agree, hence a bit more complicated strategy is needed.}

%Here and in the following, when we say ``excursion'' we always mean an excursion outside of $D_1$ and in $K\cap K'$ for either $(v_n)$ or $(v'_n)$, and if we say that it is started at $e$ we mean that the excursion starts on the only endpoint of $e$ that is in $D_1^c$.

%Given that $K$ and $K'$ are one-ended, by the assumption of \Cref{rem: weaker condition smallness instead of finiteness} and by the stationarity of the SDRW \alert{right?}, the proportion of time spent in $K \bigtriangleup K'$ before time $n$ goes to 0 when $n$ goes to $\infty$.

Call a pair of edges $(e_1,e_2)$, where $e_1$ and $e_2$ are both edges with one endpoint in $D_1$ and one endpoint outside of it, an \textbf{excursion type}: an excursion of a SDRW has this type if it starts through $e_1$, ends through $e_2$ and stays inside $K\setminus D_1=K'\setminus D_1$ otherwise. SDRW on $K$ and on $K'$ spend a 0 upper density amount of time in $D_1$ because $D_1$ is finite and SDRW on an infinite graph is null-recurrent. The parts of these two SDRW's that are outside of $D_1$ consist of excursions whose distributions only depend on their types, and hence these distributions are the same for both SDRW's. But the frequency of a certain type may differ depending on whether we take SDRW on $K$ or on $K'$. There are finitely many excursion types, hence there is an $\e>0$ and types $\tau$ and $\tau'$ such that at least an $\e$ upper density of steps of SDRW on $K$ is inside an excursion of type $\tau$  almost surely, and at least an $\e$ upper density of steps of SDRW on $K'$ is inside an excursion of type $\tau'$  almost surely. Note furthermore that every type occurs infinitely often in both SDRW's almost surely, by the properties of $D_1$.

%For any excursion type $\tau$, we define its \textbf{frequency} for the {\color{purple}SDRW $(v_n)$ in $\omega$} as $\limsup_{N\rightarrow \infty} |\{n<N\mid\text{$v_n$ is in a $\tau$ excursion}\}|/N$.
%Note that this depends on $\eta$.
%One can define a finite-state Markov chain $\M$ on the excursion types: the transition probability from $(e_1,e_2)$ to $(f_1,f_2)$ is the probability that SDRW on $K$ started from the endpoint of $e_2$ in $K\setminus D_1$ first exits $K\setminus D_1$ through $f_1$ and then makes an excursion of type $(f_1,f_2)$. We define a Markov chain $\M'$ similarly, with $K$ replaced by $K'$ everywhere. By properties 1, 2 and 3 above, the transition probability between any two states is positive both in $\M$ and in $\M'$. Let $\delta>0$ be the minimum of these transition probabilities.

We will use
\Cref{lem:technical extended version}, with the coupling between $\eta$ and $\eta'$ coming from the assumptions of our proposition. The necessary coupling of a SDRW $(v_n)$ in $\omega$ started at $x$ and a SDRW $(v'_n)$ in $\omega'$ started at $x'$ will be constructed next. 
First let $(w_n)$ be SDRW in $K$ started at $x$ and $(w'_n)$ be an independent SDRW in $K'$ started at $x'$. There is an $N(1)$ so that the proportion of time that $(w_1,\ldots,w_{N(1)})$ spends in type $\tau$ excursions is greater than $\e$. Choose such an $N(1)$ so that $w_{N(1)}$ is the last step of the $\ell(1)$'th excursion of type $\tau$, with some $\ell(1)$.
For $j=1,2\ldots, \ell(1)$, take the $j$th excursion of type $\tau$ in $(w'_n)$, and replace it with the $j$'th excursion of type $\tau$ in $(w_n)$. (We keep notation $(w'_n)$ for the sequence resulting after the replacements, and note that it is still distributed as SDRW in $\omega'$.) Now, the first $\ell(1)$ excursions of type $\tau$ are the same in $(w_n)$ and in  $(w'_n)$. We define the first few values of $\psi$ by listing the vertices $w_{\psi(1)},w_{\psi(2)},\ldots,w_{\psi (M(1))}$ in these excursions in their consecutive order (with some $M(1)$ which is the total number of vertices visited by these $\ell(1)$ excursions). Then define $\psi'(i)$ as $i=1,\ldots,M(1)$ to be equal to the index of the vertex that got replaced by
$w_{\psi(i)}$ in $(w'_n)$.

Next we repeat the procedure with the roles of $(w_n)$ and $(w'_n)$ switched, and starting after indices $\psi (M(1))$ and $\psi' (M(1))$ respectively. That is,
there is an $N(2)$ such that the proportion of time that $(w'_1,\ldots,w'_{N(2)})$ spends in type $\tau'$ excursions is greater than $\e$. Choose $N(2)$ so that $w_{N(2)}$ is the last step of an excursion of type $\tau'$. Let $\ell(2)$ be the number of excursions of type $\tau'$ after time $\psi' (M(1))$. As $j=1,2\ldots, \ell(2)$, take the $j$th
of these excursions, and replace the $j$th excursion of type $\tau'$ in
$(w_n)$ after time $\psi (M(1))$
by it. As $i=M(1)+1,\ldots,M(2)$ (with some $M(2)$ which is the total number of vertices visited by the mentioned excursions), define $\psi'(i)$ (respectively $\psi(i)$) to be the index of the $i-M(1)$'th vertex in the consecutive listing of these type $\tau'$ excursions in $(w'_n)$ (respectively $(w_n)$).

Continue alternatingly, in countably many steps, always redefining $(w_n)$ at even steps and $(w'_n)$ at odd steps, on later and later subsegments. This will always keep their distributions unchanged. Call the limiting sequences $(v_n)$ and $(v'_n)$. The functions $\psi$ and $\psi'$ list the places in $(v_n)$ and respectively $(v'_n)$ where an excursion got identified in the two sequences, hence
in \Cref{lem:technical extended version} the condition on the rooted isomorphisms around $v_{\psi(n)}$ and $v'_{\psi'(n)}$ is satisfied. Smallishness holds because of null-recurrence. Finally, $\limsup_{n\rightarrow \infty}\frac n{\psi(n)} >\e$ and $\limsup_{n\rightarrow \infty}\frac n{\psi'(n)} >\e$, because of how the $M(i)$ and $N(i)$ were chosen.
\end{proof}

\begin{remark}
    For the one-ended case, we could also have taken inspiration from \cite{hutchcroft2017indistinguishability} by taking a sequence of uniform points in a hyperfinite exhaustion instead of a SDRW.
\end{remark}

\section{Applications\label{sec:proofs}}

\subsection{The interchange process}

\begin{proof}[Proof of \Cref{t.interchange}]
    Assume for contradiction that there exists a component property $\A$ such that $\A$ and $\lnot\A$ type infinite cycles coexist with positive probability.
    In particular, there exist two vertices $x$ and $y$ such that the event $E_0$ that their cycles are infinite and $x$ has type $\A$ and $y$ type $\lnot\A$ has positive probability.
    Choose a minimal length deterministic path $P$ between $x$ and $y$.
    As every clock $\psi_e$ rings almost surely finitely many times, there exists $\rho>0$ such that,  conditionally on $E_0$, the event that for every edge $e$ with at least one endpoint on $P$, the clock $\psi_e$ does not ring during the time interval $[\beta-\rho,\beta]$ has positive probability.
    We write $E_1$ for its intersection with $E_0$.

    Let $(\psi^1_e)$ be i.i.d.\ unit-rate Poisson point processes indexed by the edges of $G$ and $\pi_1$ the induced permutation.
    We then define $(\psi^2_e)$ by $\psi^2_e:=\psi^1_e$ for $e$ not in $P$;
    for $e$ in $P$, $$\psi^2_e(t):=\begin{cases}
        \psi^1_e(t),&\text{if } t<\beta-\rho,\\
        \psi^1_e(\beta-\rho)+\xi_e(t-\beta+\rho),&\text{else,}
    \end{cases}$$
where $\xi_e$ is a unit-rate Poisson process, independent of everything else.
Write $\pi_2$ for the permutation induced by $\psi^2$.

In essence, $\psi^2$ is $\psi^1$ with the clocks on the edges of $P$ resampled independently in the time interval $[\beta-\rho,\beta]$.
Conditionally on $\psi^1$ being in $E_1$, by independence, there is a positive probability that the particles on $x$ and $y$ at time $\beta-\rho$ are swapped at time $\beta$, while every other particle on a vertex of $P$ at time $\beta-\rho$ is still there at time $\beta$.
If this is the case, then $\pi_2=\tau_{x,y}\circ\pi_1$, where $\tau_{x,y}$ is the transposition that swaps $x$ and $y$.

Let $v_n:=\pi_1^n(y)$ and $v'_n:=\pi_2^n(y)$ for $n\in\N$.
They are both stationary sequences, by \Cref{lem: random permutation and SDRW are stationary}.
Write $\W$ for the event that $\psi^1$ is in $E_1$ and that $\pi_2=\tau_{x,y}\circ\pi_1$.
On this event, $(v_n)$ and $(v'_n)$ are equal: if $\pi^n_1(y)$ was equal to $x$ or $y$ for some $n$, then either the cycle of $y$ in $\pi_1$ would be finite or it would contain $x$.
Furthermore, $\psi^1$ and $\psi^2$ differ only in $P$, which is small for $(v_n)$ and $(v'_n)$ on $\W$ (as the cycle of $y$ is infinite).
We thus can use \Cref{lem:technical lemma for indis} to show that $y$ has the same $\A$-type in $\psi^1$ and $\psi^2$.

On the other hand, let $w_n:=\pi_1^{-n}(x)$ and $w'_n:=\pi_2^{-n}(y)$ for $n\in\N$.
They are also both stationary sequences.
On $\W$, $(w_n)$ and $(w'_n)$ are asymptotically equal as $w'_{n}=w_n$ for every $n\geq1$:
$\pi_2^{-1}(y)=\pi_1^{-1}\circ\tau_{x,y}(y)=\pi_1^{-1}(x)$, and by the same reasoning as above, $\pi_1^{-n}(x)\notin\{x,y\}$ for every $n\geq1$.
Once again, $P$ is small for  $(w_n)$ and $(w'_n)$ conditionaly on $\W$ and we can use \Cref{lem:technical lemma for indis} to show that $x$ in $\psi^1$ and $y$ in $\psi^2$ have the same $\A$-type.

Combining these two results, $x$ and $y$ have the same $\A$-type in $\psi^1$ conditionally on $\W$, contradicting its definition.
\end{proof}

\subsection{The loop $O(n)$ model}
A {\bf  loop configuration} on $G$ is any subgraph of $G$ in which every vertex has degree 0 or 2.
Such a subgraph is composed of a disjoint mixture of bi-infinite paths and closed, non-intersecting finite length loops.

For $F$ a finite subgraph of $G$, $\xi$ a loop configuration on $G$ and two parameters $x>0,n>0$, we define the following measure on the set of loop configurations $\omega$ that coincide with $\xi$ outside of $F$:
$$
\P^{n,x}(\omega):=\frac{x^{|\omega|}n^{l(\omega)}}{Z^{\xi,F}_{n,x}}
$$
where $|\omega|$ is the number of edges in $\omega\cap F$, $l(\omega)$ is the number of connected components of $\omega$ intersecting $F$ and $Z^{\xi,F}_{n,x}$ is the unique constant making $\P^{\xi,F}_{n,x}$ a probability measure.

If $\omega$ is a fixed configuration coinciding with $\xi$ outside of $F$ and $\gamma$ a finite loop in $F$ of length $k$ such that $\omega \bigtriangleup \gamma$ is still a loop configuration,
$$
\P^{\xi,F}_{n,x}(\omega\bigtriangleup\gamma)=\P^{\xi,F}_{n,x}(\omega)x^{|\omega\bigtriangleup\gamma|-|\omega|}n^{l(\omega\bigtriangleup\gamma)-l(\omega)}.
$$

If we flip in $\omega$ the edges of $\gamma$ one by one, each flip adds or removes an edge and can at most add or remove a connected component (note here that at least half of these intermediate steps won't be loop configurations).
This implies that $-k\leq |\omega\bigtriangleup\gamma|-|\omega|\leq k$ and $-k\leq l(\omega\bigtriangleup\gamma)-l(\omega)\leq k$ and 
$$
\min(x,1/x)^k\min(n,1/n)^k\leq\frac{\P^{\xi,F}_{n,x}(\omega\bigtriangleup\gamma)}{\P^{\xi,F}_{n,x}(\omega)}\leq \max(x,1/x)^k\max(n,1/n)^k.
$$
These bounds only depend on the length of $\gamma$ and the fixed parameters $x$ and $n$.

Let $\omega$ have law $\P_{n,x}$, any translation-invariant infinite-volume Gibbs measure on the loop configurations on $G$ that is the thermodynamic limit of $\P^{\xi,F_m}_{n,x}$ when $m$ goes to infinity, $F_m$ being an exhaustion of $G$.
Let then $\gamma$ be a finite loop and $\tilde \omega$ be a loop configuration constructed by taking $\omega$ and resampling independently the edges in $\gamma$ such that it too has law $\P_{n,x}$.
Then conditionally on the event that $\omega\bigtriangleup\gamma$ is a loop configuration, there is a positive probability that $\tilde\omega$ is exactly $\omega\bigtriangleup\gamma$.

\begin{lemma}
    \label{lem: grafting in O(n)}
    Let $G$ be an amenable transitive graph, $\A$ a component property and $\omega$ a random loop configuration with law $\P_{n,x}$.
If infinite component of type $\A$ and $\lnot\A$ s can coexist with positive probability, then
there exists with positive probability a collection of finite loops $\gamma_0,\ldots,\gamma_k$ (not necessarily contained in $\omega$) and two biinfinite paths $p_1$ and $p_2$ contained in $\omega$ with type $\A$ and $\lnot\A$ such that $\omega\bigtriangleup\gamma_0\bigtriangleup\ldots\bigtriangleup\gamma_k$ is a loop configuration with a connected component having infinite intersections with both $p_1$ and $p_2$.
\end{lemma}

\begin{proof}
The main idea of this proof is that if we remove every finite loop from $\omega$, and then add finite disjoint paths between nearby infinite paths with the opposite $\A$-type in an automorphism-invariant way, then there are two options for the resulting invariant percolation $\omega'$. One possibility is that we created cycles in $\omega'$, in which case we take one of them to be $\gamma_0$, and we let $\gamma_1,...,\gamma_k$ be the finite loops in $\omega$ that $\gamma_0$ crosses. The symmetric difference $\omega\bigtriangleup\gamma_0\bigtriangleup\ldots\bigtriangleup\gamma_k$ removes $\gamma_1,...,\gamma_k$ then takes the symmetric difference of the infinite paths with $\gamma_0$, creating the desired loop configuration.
The other possibility would be that we did not create any cycles, meaning that $\omega'$ is an invariant forest in which every vertex has a positive probability to have degree 3, which is in contradiction with the amenability of $G$ following the classic Burton-Keane argument \cite{burtonkeane}.

\begin{figure}[htbp]
    \centering
    \includegraphics[width=0.8\textwidth]{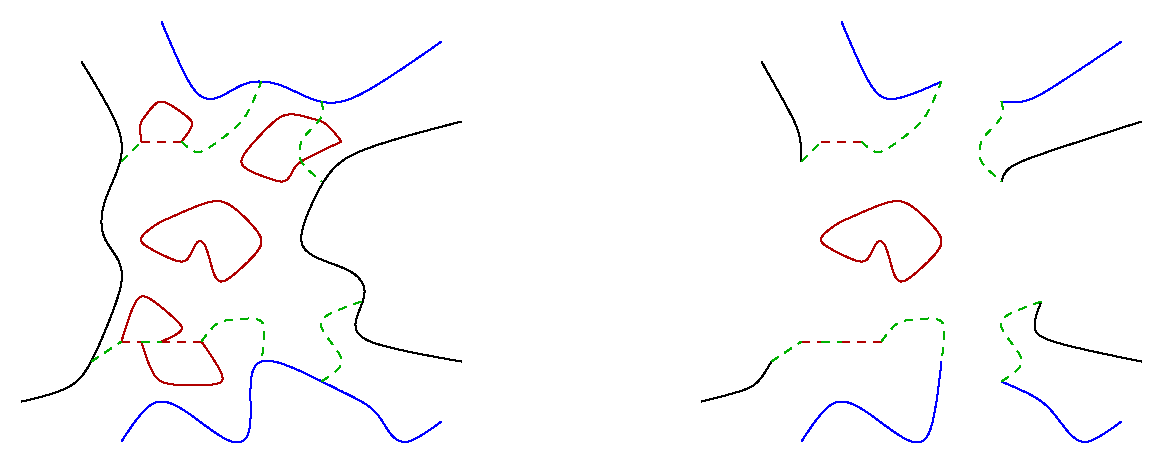}
    \caption{Creating $\omega\bigtriangleup\gamma_0\bigtriangleup\ldots\bigtriangleup\gamma_k$: infinite paths in $\omega$ with different $\A$-types are black and blue, respectively, finite loops are dark red, the disjoint finite connecting paths are dashed.}
    \label{fig:O(n)Amen}
\end{figure}

Let us start by building $\omega'$.
Let $(\xi_v)_{v\in G}$ be i.i.d.\ random variables, uniform in $[0,1]$, and let $d$ be the minimal distance between a vertex in an infinite path of type $\A$ and a vertex in an infinite path of type $\lnot\A$.
For every pair $x,y$ of two vertices in $G$, define $Q(x,y)\subset E(G)$ as follows: if
\begin{itemize}
    \item $x$ has type $\A$, $y$ has type $\lnot\A$, 
    \item $\xi_x>\xi_{x'}$ and $\xi_y>\xi_{y'}$ for any $x'\in B(x,3d)$ of type $\A$ and any $y'\in B(y,3d)$ of type $\lnot\A$,
    \item the distance between $x$ and $y$ in $G$ is $d$,
\end{itemize}
then take $Q(x,y)$ to be the set of edges of a uniformly and independently chosen path of length $d$ between $x$ and $y$, and the empty set else.

Define $\omega':= (\omega\setminus L)\cup \bigcup_{x,y\in G}Q(x,y)$, where $L$ is the set of edges that are in a finite loop in $\omega$.
By construction, $\omega'$ has no finite components.
Furthermore, for any $x$ in $G$, there is clearly almost surely at most one $y$ such that $Q(x,y)$ is nonempty, and for every nonempty $Q(x,y)$, it intersects almost surely only two infinite components of $\omega$, each at exactly one vertex --- otherwise, there would be a path of length strictly less than $d$ between an $\A$ and a $\lnot\A$ cluster, contradicting the definition of $d$.
This means that every vertex in $\omega'$ has degree at most 3.

We will first prove that if $\omega'$ admits a finite cycle, then we can take $\gamma_0$ equal to such a cycle and $\gamma_1,\ldots,\gamma_k$ to be all the finite loops in the original configuration $\omega$ that have at least one vertex in common with $\gamma_0$, and that this collection of loops satisfies the statement of the lemma.
It is easy to see that $\bar\omega:=\omega\bigtriangleup\gamma_0\bigtriangleup\ldots\bigtriangleup\gamma_k$ is in fact a loop configuration.

We now show there are two infinite paths $p_1$ and $p_2$ with types $\A$ and $\lnot \A$ such that there is an infinite path in $\bar \omega$ with infinite intersections with both $p_1$ and $p_2$.
We can assume without loss of generality that $\gamma_0$ goes through every infinite path in $\omega$ at most once --- i.e., that for every infinite path $p$ in $\omega$, the intersection $\omega\cap p$ is connected.
Indeed, if $p = \cdots,v_{-1},v_0,v_1,\cdots$, we can find a maximal $i_M$ such that $\gamma_0$ contains $v_{i_M}$.
Then we can define $\gamma_0'$ by following $\gamma_0$ from $i_M$ in the direction that goes away from $p$ until it hits $p$ again, at which point $\gamma_0'$ can follow $p$ in the appropriate direction until it hits $i_M$.
We then take $\gamma_0'$ in place of $\gamma_0$.
This procedure can be repeated until $\gamma_0$ goes through every infinite path at most once, as it gets strictly shorter every time the procedure is applied. 
%so it has to terminate after finitely many steps.

So, there exists a $p_1 = \cdots,v_{-1},v_0,v_1,\cdots$ and $i_M$ such that, as before, $i_M$ is the highest value of $i$ such that $v_i$ is in $\gamma_0$.
If we follow the component of $v_{i_M}$ in $\bar\omega$, it is equal to $v_{i_M+1},v_{i_M+2},\cdots$ in one direction, and in the other, it goes through a finite number of edges that are not in an infinite component of $\omega$ before reaching $p_2$, another infinite component of $\omega$, and then following it for the rest of its existence.

By construction of $\omega'$, $p_1$ and $p_2$ have different $\A$-types, and the component of $v_{i_M}$ in $\bar\omega$ has infinite intersections with both $p_1$ and $p_2$, as desired. Thus, to prove our lemma, we are left to prove that $\omega'$ indeed contains a cycle almost surely.

Suppose for contradiction that $\omega'$ does not contain a cycle.
The argument is similar to the classic Burton-Keane argument for the uniqueness of the infinite cluster in Bernoulli percolation on amenable transitive graphs \cite{burtonkeane}.
We call a vertex $v$ a trifurcation if its connected component $C$ in $\omega'$ satisfies: $C\setminus\{\{v,v'\}\mid v'\sim_Gv\}$ has exactly three connected components.
If we assume for contradiction that $\omega'$ does not contain any cycle with positive probability, then there are two possibilities: either every connected component of $\omega'$ is an infinite path of $\omega$ with probability 1, or $\omega'$ contains a trifurcation with positive probability.

We can easily rule out the first option: by definition, there is a positive probability for $p$ and $p'$, two infinite paths of $\omega$ of type $\A$ and $\lnot \A$, respectively, to be at distance $d$ from one another.
By independence, there is a positive probability that when constructing $\omega'$, we add a finite path starting from $p$ and ending in $p'$.
The endpoints of said finite path are both trifuractions in $\omega'$, if $\omega'$ does not contain any cycle.

Then, because $G$ is countable and transitive, every vertex has a probability $\delta>0$ to be a trifurcation in $\omega'$, which is impossible in an amenable transitive graph by the Burton-Keane argument.
%Let $\Lambda_n$ be a F\o lner sequence.
%The expected number of trifurcation in $\Lambda_n$ is $\delta|\Lambda_n|$, thus if $n$ is enough for $|\partial \Lambda_n|/|\Lambda_n|$ to be smaller than $\delta/2$, the expected number of trifurcations in $\Lambda_n$ is larger than $2|\partial\Lambda_n|$.
%Hence with, positive probability, the actual number of trifuractions in $\Lambda_n$ is larger than $2|\partial\Lambda_n|$.
%We can then conclude using the fact, provided in \cite{burtonkeane}, that the number of trifurcations in any finite set can not be larger than the cardinality of its boundary.
\end{proof}

\begin{proof}[Proof of \Cref{thm: O(n) indistinguishability}]
   Assume for contradiction that $\A$ and $\lnot\A$ type clusters can coexist with positive probability.
   Then, using \Cref{lem: grafting in O(n)}, there exist vertices $x$ and $y$ and a finite set of edges $S$ (the edges of $\gamma_0,\dots,\gamma_k$) such that if $\omega$ has law $\P_{n,x}$ and $\omega'$ has law $\P^{\omega,S}_{n,x}$, there is a positive probability of the event $\W$ that $x$ has type $\A$ and $y$ has type $\lnot\A$ in $\omega$, and that the cluster of $y$ in $\omega'$ has an infinite intersection with both the clusters of $x$ and $y$ in $\omega$.
   Then all the assumptions of \Cref{cor: technical on graphs} hold, and we can apply it to deduce conditionally on $\W$ that $x$ in $\omega$ and $y$ in $\omega'$ have the same $\A$-type, and that $y$ in $\omega$ and $y$ in $\omega'$ have the same $\A$-type, hence $x$ and $y$ have the same $\A$-type conditionally on $\W$ in $\omega$, a contradiction. 
\end{proof}

\subsection{Corner percolation}\label{s.corner}
        We work on the graph $\Z^2$ with its usual lattice structure and its edge set $\mathsf E$.
Its vertices are partitioned into $\Z^2_e:=\{(x,y)\in\Z^2\mid x+y\in 2\Z\}$ and $\Z^2_o:=\{(x,y)\in \Z^2\mid x+y \in 2\Z+1\}$.

Let $\Omega:=\{-1,1\}^\Z\times \{-1,1\}^\Z$. For any  $(\xi,\eta)\in\Omega$, we define a percolation configuration $\omega \in 2^{\mathsf E}$ as follows:
\begin{itemize}
    \item a vertical edge $\{(x,y),(x,y+1)\}$ is open (i.e., in $\omega$) if and only if $(x,y)\in\Z^2_o$ and $\xi_x=1$ or if $(x,y)\in\Z^2_e$ and $\xi_x=-1$
    \item a horizontal edge $\{(x,y),(x+1,y)\}$ is open if and only if $(x,y)\in\Z^2_o$ and $\eta_x=1$ or if $(x,y)\in\Z^2_e$ and $\eta_x=-1$.
    \end{itemize}

    We choose two parameters $p,q\in(0,1)$, $(p,q)\neq (1/2,1/2)$ and
let $(\xi_n)_{n\in\Z}$ be a sequence of i.i.d.\ random variables with $\P(\xi_0=1)=p$ and $\P(\xi_0=-1)=1-p$, $(\eta_n)_{n\in\Z}$ a sequence of i.i.d.\ random variables with $\P(\eta_0=1)=q$, $\P(\eta_0=-1)=1-q$.
Corner percolation is then the configuration $\omega$ resulting from this $(\xi,\eta)$.

    It is known from \cite{cornerPerco2025} that there is then almost surely an infinite number of infinite clusters, and every infinite cluster is almost surely an infinite path with asymptotic slope $\frac {2p-1}{1-2q}$:
    \begin{theorem}[\cite{cornerPerco2025}]
        \label{thm:asymptotic slope}
        For every $\e>0$ and any $x$ in the cluster of $o$, with probability 1 there exists $R>0$ such that one of these three events is true: either the distance between $o$ and $x$ is less than $R$, or $q\neq 1/2$ and the slope of the straight line going through $o$ and $x$ is in $\big[\frac{2p-1}{1-2q}-\e, \frac{2p-1}{1-2q}+\e\big]$, or $q=1/2$ and the slope of the straight line going through $o$ and $x$ is in $(-\infty,-1/\e]\cup[1/\e,+\infty)$.
    \end{theorem}
    
    We borrow the height function from \cite{peteCornerPercolationMathbb2008a}: define two independent random walks on $\Z$ as 
    \begin{align*}
        X_0=0,\qquad \text{for }n>0,\  X_n={\sum_{i=1}^n\xi_i}, \qquad \text{and for } n<0,\  X_n=-\sum_{i=n+1}^0\xi_i;\\
        Y_0=0, \qquad  \text{for }n>0,\ Y_n={\sum_{j=1}^n\eta_j}, \qquad \text{and for } n<0,\ Y_n=-\sum_{j=n+1}^0\eta_j;
\end{align*}
and the height function $$H(n+1/2,m+1/2):=\begin{cases} \frac{X_n+Y_m} 2,&\text{if $(n,m)\in\Z^2_e$}\\
    \frac{X_n+Y_m+1}2,&\text{if $(n,m)\in\Z^2_o$}.
\end{cases}
$$

The important property of the height function that will be useful for us later is that for every cluster $\mathcal C$, there is $h\in\mathbb Z$ such that if $x\in \mathcal C\cap \Z^2_e$, then $H(x+(1/2,1/2))=h$.
We call this the height of $\C$.
Another way to see the height is to color every square of $\Z^2$ in black or white according to whether the bottom-left corner is even or odd, respectively. Then, 
two squares without an edge between them have the same height, and
when an edge of percolation separates a black and a white square, the black square has always height 1 less than the white square.

It is shown in \cite{cornerPerco2025} that (when $(p,q)\neq (1/2,1/2)$), for every $h$ in $\Z$, there exists almost surely a unique infinite cluster with height $h$ --- the height function is a bijection between $\Z$ and the infinite clusters.

\begin{figure}[hbtp]
    \centering
    \begin{subfigure}{0.45\textwidth}
    \includegraphics[width=\textwidth]{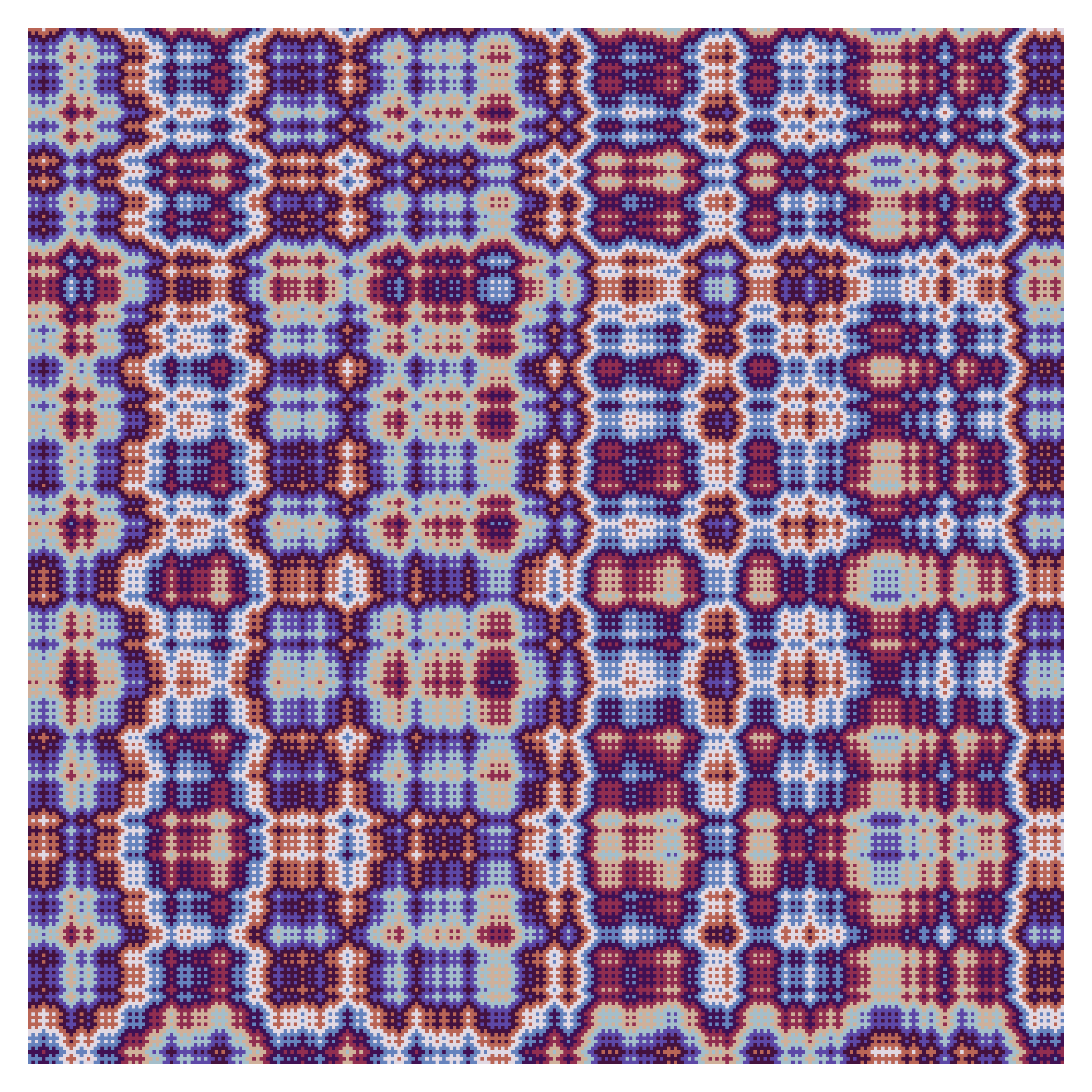}
    \end{subfigure}
    \hfill
    \begin{subfigure}{0.45\textwidth}
        \includegraphics[width=\textwidth]{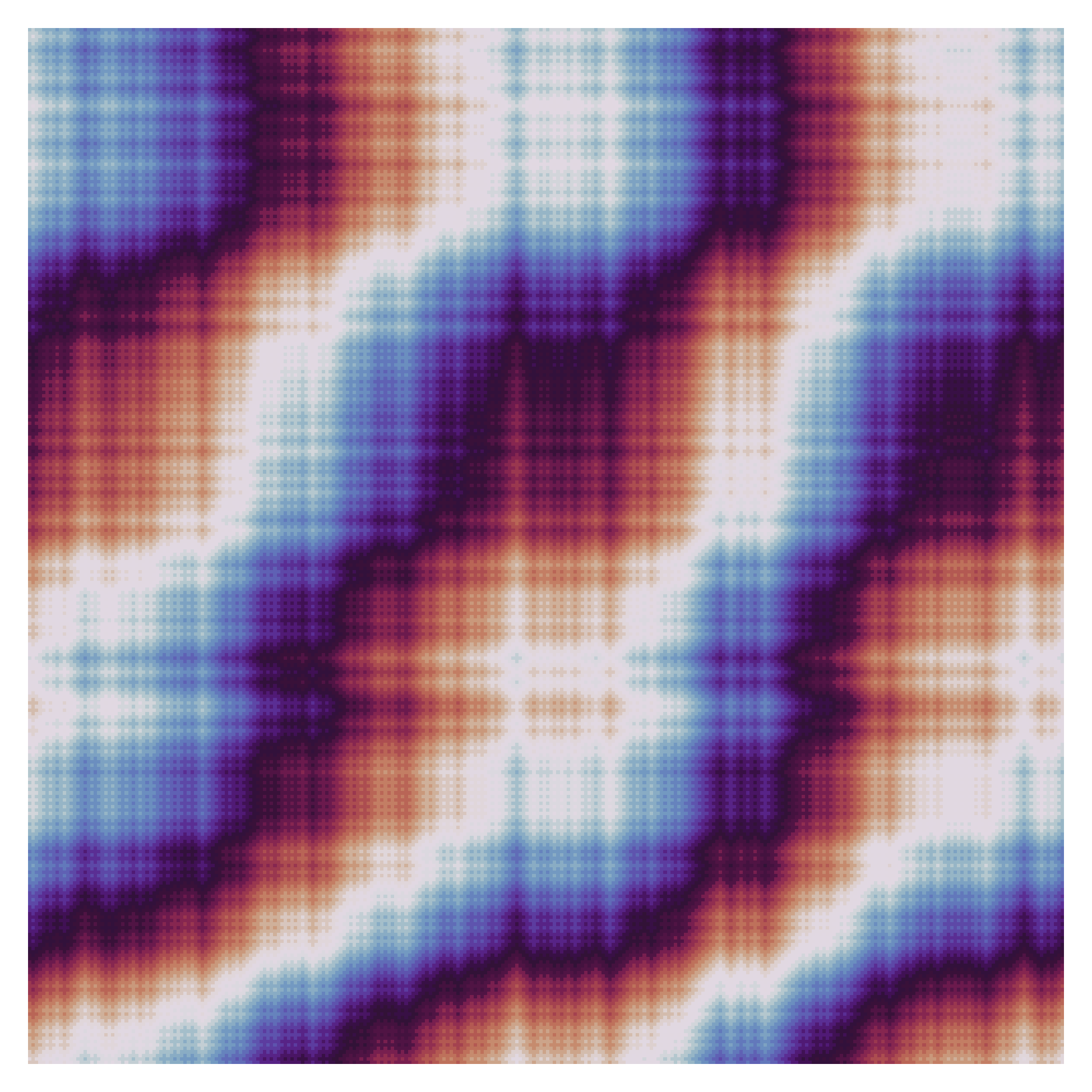}
    \end{subfigure}
    \caption{Corner percolation samples with parameters $(0.5,0.5)$ and $(0.7,0.3)$. Two vertices with the same height have the same color.}
\end{figure}

We deduce from this the following lemma, which will be useful later. In the rest of this section, we write $a_1$ and $a_2$ for the two coordinates of a vertex $a$.
%\begin{lemma}
%\label{cor: parity is recoverable}
 %    If $p\neq 1/2$ (resp.\ $q\neq 1/2$), and we say that the positive direction in a cluster is the one in which the first coordinate (resp.\ the second coordinate) goes to $+\infty$, then for every even vertex $v$ that belongs to an infinite cluster, the edge going in the positive direction out of $v$ is almost surely horizontal if $p>1/2$ (resp.\ $q<1/2$) and almost surely vertical if $p<1/2$ (resp.\ $q>1/2$).
%\end{lemma}
\begin{lemma}
\label{cor: parity is recoverable}
     If $q\neq 1/2$, and we say that the positive direction in a cluster is the one in which the first coordinate goes to $+\infty$, then for every even vertex $v$ that belongs to an infinite cluster, the edge going in the positive direction out of $v$ is almost surely horizontal if $q>1/2$ and almost surely vertical if $q<1/2$.
\end{lemma}

\begin{proof}
        Let us prove the case $q>1/2$; the other case is easily handled by symmetry.
        The condition $q\neq 1/2$ guarantees that every infinite cluster almost surely crosses the $y$ axis.

        Let us take a cluster $\C$ and label its edges $(\dots ,e_{-1},e_0,e_1,\dots)$, writing $a^i$ and $b^i$ for the endpoints of $e_i$, both labellings chosen such that \begin{itemize}
            \item $a^{i+1}=b^i$,
            \item $\lim_{i\rightarrow \infty} b^i_1 = \lim_{i\rightarrow \infty} a^i_1=\infty$.
        \end{itemize}
        Let $\tau$ be the smallest $i$ such that $b^i_1=1$ (it exists thanks to \Cref{thm:asymptotic slope}).
        It is clear that $e_\tau$ is horizontal, and that $a^\tau_1=0$.  
        
        We will now show that $a_2^\tau=b^\tau_2$ is even, and thus $a^\tau$ is even, proving the statement of the lemma for $a^\tau$. The rest of the lemma follows from an easy induction on $a^i$.

        Let $\C'$ be another infinite cluster, with edges $(\dots,f_{-1},f_0,f_1,\dots)$ labelled like in $\C$, where $f_i$ has endpoints $c^i,d^i$, and $c^{i+1}=d^i$, and $c^i_1$ and $d^i_1$ go to $+\infty$ with $i$.
        Define $\tau'$ like $\tau$ as the smallest $i$ such that $d^i_1=1$.
        Note that $a^\tau_2$ and $c^{\tau'}_2$ have the same parity.
        Indeed, if it was not the case, one could construct an infinite cluster in which both directions have first coordinates going to $-\infty$ by changing finitely many $\xi_i$, contradicting \Cref{thm:asymptotic slope}, as demonstrated in \Cref{fig: e_tau parity}.

\begin{figure}[hbtp]
    \centering
     \begin{subfigure}[t]{0.3\textwidth}
    \def\svgwidth{\textwidth}
    %% Creator: Inkscape 1.4 (e7c3feb100, 2024-10-09), www.inkscape.org
%% Accompanies image file 'bijection1.pdf' (pdf, eps, ps)
\begingroup%
  \makeatletter%
  \providecommand\color[2][]{%
    \errmessage{(Inkscape) Color is used for the text in Inkscape, but the package 'color.sty' is not loaded}%
    \renewcommand\color[2][]{}%
  }%
  \providecommand\transparent[1]{%
    \errmessage{(Inkscape) Transparency is used (non-zero) for the text in Inkscape, but the package 'transparent.sty' is not loaded}%
    \renewcommand\transparent[1]{}%
  }%
  \providecommand\rotatebox[2]{#2}%
  \newcommand*\fsize{\dimexpr\f@size pt\relax}%
  \newcommand*\lineheight[1]{\fontsize{\fsize}{#1\fsize}\selectfont}%
  \ifx\svgwidth\undefined%
    \setlength{\unitlength}{675bp}%
    \ifx\svgscale\undefined%
      \relax%
    \else%
      \setlength{\unitlength}{\unitlength * \real{\svgscale}}%
    \fi%
  \else%
    \setlength{\unitlength}{\svgwidth}%
  \fi%
  \global\let\svgwidth\undefined%
  \global\let\svgscale\undefined%
  \makeatother%
  \begin{picture}(1,1)%
    \lineheight{1}%
    \setlength\tabcolsep{0pt}%
    \put(0,0){\includegraphics[width=\unitlength,page=1]{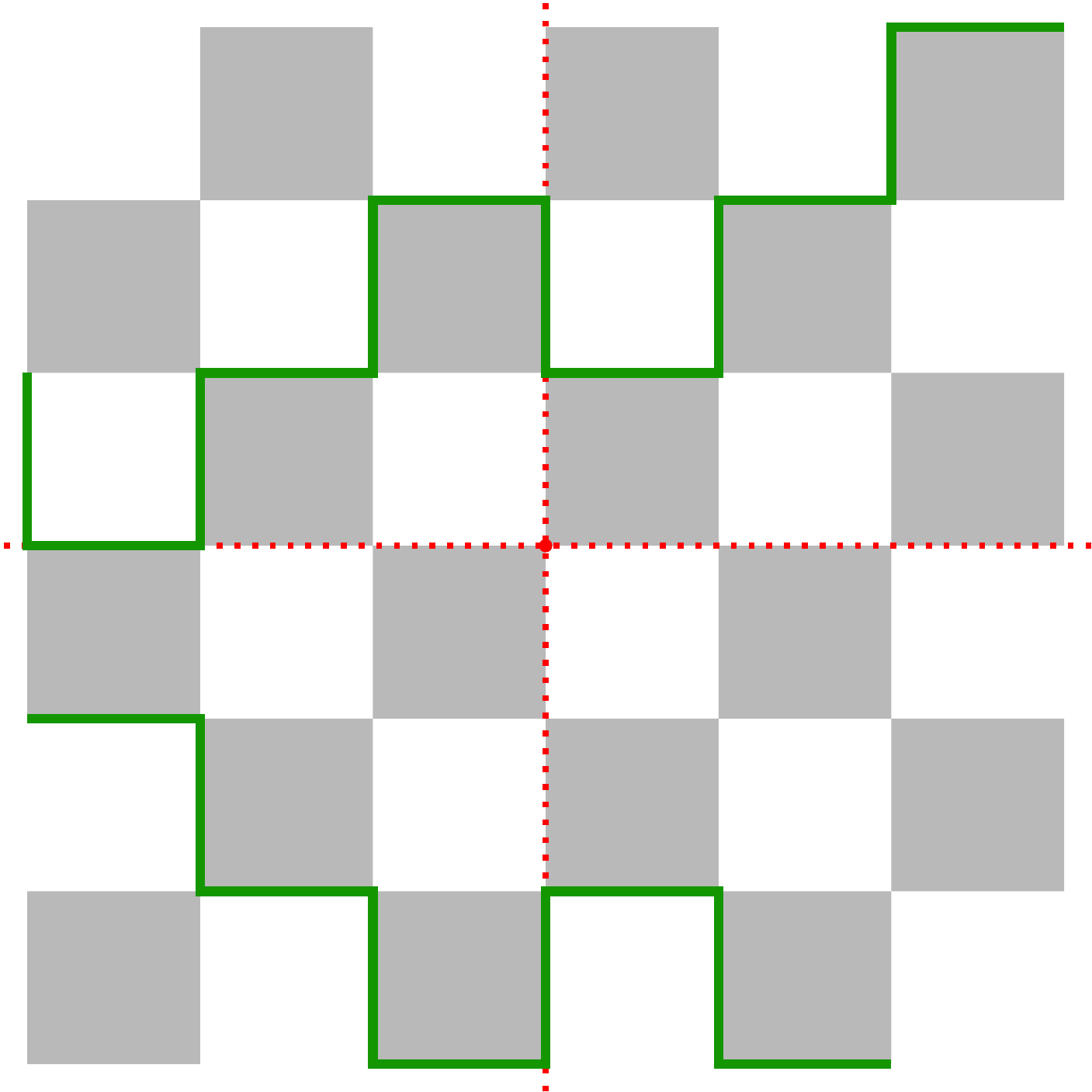}}%
    \put(0.50779049,0.51313358){\color[rgb]{0.75686275,0,0}\makebox(0,0)[lt]{\lineheight{1.25}\smash{\begin{tabular}[t]{l}$o$\end{tabular}}}}%
    \put(0,0){\includegraphics[width=\unitlength,page=2]{Images/bijection1.pdf}}%
    \put(0.56,0.67851454){\makebox(0,0)[lt]{\lineheight{1.25}\smash{\begin{tabular}[t]{l}$e$\\\end{tabular}}}}%
    \put(0.56,0.20556369){\makebox(0,0)[lt]{\lineheight{1.25}\smash{\begin{tabular}[t]{l}$e'$\\\end{tabular}}}}%
  \end{picture}%
\endgroup%

       \end{subfigure}
    \hglue 2 cm
    \begin{subfigure}[t]{0.3\textwidth}
        \def\svgwidth{\textwidth}
        %% Creator: Inkscape 1.4 (e7c3feb100, 2024-10-09), www.inkscape.org
%% Accompanies image file 'bijection2.pdf' (pdf, eps, ps)
\begingroup%
  \makeatletter%
  \providecommand\color[2][]{%
    \errmessage{(Inkscape) Color is used for the text in Inkscape, but the package 'color.sty' is not loaded}%
    \renewcommand\color[2][]{}%
  }%
  \providecommand\transparent[1]{%
    \errmessage{(Inkscape) Transparency is used (non-zero) for the text in Inkscape, but the package 'transparent.sty' is not loaded}%
    \renewcommand\transparent[1]{}%
  }%
  \providecommand\rotatebox[2]{#2}%
  \newcommand*\fsize{\dimexpr\f@size pt\relax}%
  \newcommand*\lineheight[1]{\fontsize{\fsize}{#1\fsize}\selectfont}%
  \ifx\svgwidth\undefined%
    \setlength{\unitlength}{675bp}%
    \ifx\svgscale\undefined%
      \relax%
    \else%
      \setlength{\unitlength}{\unitlength * \real{\svgscale}}%
    \fi%
  \else%
    \setlength{\unitlength}{\svgwidth}%
  \fi%
  \global\let\svgwidth\undefined%
  \global\let\svgscale\undefined%
  \makeatother%
  \begin{picture}(1,1)%
    \lineheight{1}%
    \setlength\tabcolsep{0pt}%
    \put(0,0){\includegraphics[width=\unitlength,page=1]{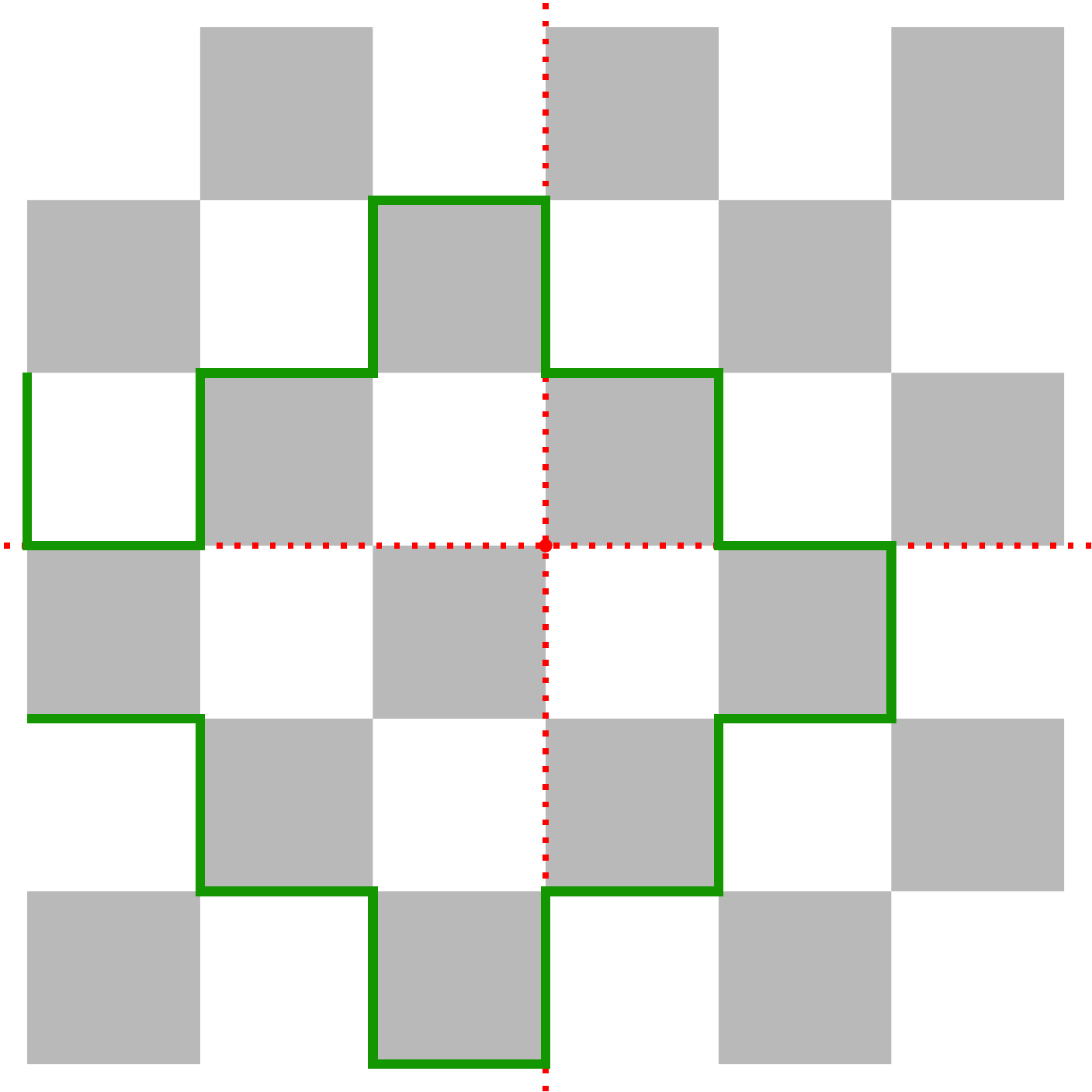}}%
    \put(0.50779049,0.51313358){\color[rgb]{0.75686275,0,0}\makebox(0,0)[lt]{\lineheight{1.25}\smash{\begin{tabular}[t]{l}$o$\end{tabular}}}}%
    \put(0,0){\includegraphics[width=\unitlength,page=2]{Images/bijection2.pdf}}%
    \put(0.56,0.67851454){\makebox(0,0)[lt]{\lineheight{1.25}\smash{\begin{tabular}[t]{l}$e$\\\end{tabular}}}}%
    \put(0.56,0.20556369){\makebox(0,0)[lt]{\lineheight{1.25}\smash{\begin{tabular}[t]{l}$e'$\\\end{tabular}}}}%
    \put(0,0){\includegraphics[width=\unitlength,page=3]{Images/bijection2.pdf}}%
  \end{picture}%
\endgroup%

    \end{subfigure}
    
   \caption{\label{fig: e_tau parity} If there were two infinite components (in green in the left picture), crossing from $\Z_-\times\Z$ to $\Z_+\times \Z$ in the positive sense at $e$ and $e'$, respectively, then, because the $y$-coordinates $-2$ and $1$ have different parities, one could change only $\xi_1$ and $\xi_2$ to create a green component (in the right picture) with both directions going towards $(-\infty,-\infty)$, contradicting \Cref{thm:asymptotic slope}.}
\end{figure}

Let us now look at the set $I$ of all edges of the form $\{(0,n),(1,n)\}$.
As noted in \cite{cornerPerco2025}, infinite clusters have an odd number of edges in $I$, while finite clusters have an even number of edges in $I$.
Furthermore, if $\F$ is a finite cluster, and $t,t'$ are the centers of squares adjacent to $\F$ on the outside of the loop created by $\F$, then $H(t)=H(t')$ by definition.
This means that if $e$ is an edge in $I$ that belongs to an infinite cluster, we only need to know of the number of edges in infinite clusters in $I$ to know its height.
More precisely, if one ascends $I$ edge by edge starting at $\{(0,0),(1,0)\}$, every finite cluster contributes an even number of times to the height, once $+1$ and once $-1$, which in total amounts to 0. On the other hand, any infinite cluster divides the plane into an ``upper'' and a ``lower'' part, with the squares adjacent to the cluster on either side all having the same height, differing from each other by 1. Furthermore, the parity argument in the previous paragraph implies that depending on the parity of $a^{\tau}_2$ (which does not depend on the choice of $\C$), either the height goes to $+\infty$ or to $-\infty$ when starting at $\C$ and considering the sequence of infinite clusters immediately ``above'' $\C$, as they either all add 1 to the height or all remove 1 from it. In particular, if the lemma was wrong, then the height would go to $-\infty$, contradicting the fact that this sequence of heights is exactly a subsequence of $X_i$, which, as a biased random walk with parameter $q>1/2$, goes to $+\infty$.
    \end{proof}

    \begin{theorem}
        \label{thm: corner percolation has indistinguishable clusters}
        Corner percolation on $\Z^2$ with parameters $(p,q)\neq (\frac 1 2,\frac 1 2)$ has indistinguishable clusters.
    \end{theorem}
    \begin{proof}
        Let us assume for contradiction that there exists a component property $\A$ such that $\A$ and $\lnot\A$ clusters coexist with positive probability.
        Let us also assume that $\frac{2p-1}{1-2q}$ (the asymptotic slope of the infinite clusters as given by \Cref{thm:asymptotic slope}) is in $[0,1]$ and that $q<1/2
        ,\ p\geq1/2$; the proof in every other case is given by the numerous symmetries of the model.
        Here again, the positive direction of an infinite cluster is the one for which the first coordinate goes to $+\infty$.

        \Cref{thm:asymptotic slope} ensures that there is a positive probability for $o=(0,0)$ to be in an infinite path of type $\A$ and that the half-path starting from $o$ in the negative direction lies entirely in $(\infty,-1]\times\Z$.
        Call this event $E_0$.
        By assumption, there exists a second vertex $b$ such that conditionally on $E_0$, there is a positive probability for $b$ to be in an infinite path of type $\lnot\A$ and for the half-path starting from $b$ in the positive direction to lie entirely in $[b_1+1,\infty)\times\Z$.
        Let us write $E$ for the event that both $E_0$ and this event about $b$ happens.
        
        Furthermore, by unimodularity, there are almost surely infinitely many vertices with this same property as $b$ with the same second coordinate as $b$, so we can without loss of generality assume that $b_1,b_2>0$ and that $b_1>b_2$ (see \Cref{fig: corner_surgery}).

        Now, let $(\xi_i)_{i\in\Z}$ and $(\eta_j)_{j\in\Z}$ be sequences of i.i.d.\ random variables uniform in $\{-1,1\}$, and let $\omega$ be the corresponding corner percolation.

        Then, build $(\xi'_i)_{i\in\Z}$ by $\xi'_i:=\xi_i$ for $i\notin [0,b_1]$, and i.i.d.\ uniform in $\{-1,1\}$ else.
        Call $\omega'$ the corner percolation deduced from $\xi'$ and $\eta$.

        Now we will show that there is a positive probability for $\omega$ to be in $E$ while $\omega'$ having a cluster that contains the $\omega$-half-path started from $o$ in the negative direction and the $\omega$-half-path started from $b$ in the positive direction.
\begin{figure}[hbtp]
    \centering
    \begin{subfigure}[t]{0.3\textwidth}
    \def\svgwidth{\textwidth}
    %% Creator: Inkscape 1.4 (e7c3feb100, 2024-10-09), www.inkscape.org
%% Accompanies image file 'corner_surgery.pdf' (pdf, eps, ps)
\begingroup%
  \makeatletter%
  \providecommand\color[2][]{%
    \errmessage{(Inkscape) Color is used for the text in Inkscape, but the package 'color.sty' is not loaded}%
    \renewcommand\color[2][]{}%
  }%
  \providecommand\transparent[1]{%
    \errmessage{(Inkscape) Transparency is used (non-zero) for the text in Inkscape, but the package 'transparent.sty' is not loaded}%
    \renewcommand\transparent[1]{}%
  }%
  \providecommand\rotatebox[2]{#2}%
  \newcommand*\fsize{\dimexpr\f@size pt\relax}%
  \newcommand*\lineheight[1]{\fontsize{\fsize}{#1\fsize}\selectfont}%
  \ifx\svgwidth\undefined%
    \setlength{\unitlength}{675bp}%
    \ifx\svgscale\undefined%
      \relax%
    \else%
      \setlength{\unitlength}{\unitlength * \real{\svgscale}}%
    \fi%
  \else%
    \setlength{\unitlength}{\svgwidth}%
  \fi%
  \global\let\svgwidth\undefined%
  \global\let\svgscale\undefined%
  \makeatother%
  \begin{picture}(1,1)%
    \lineheight{1}%
    \setlength\tabcolsep{0pt}%
    \put(0,0){\includegraphics[width=\unitlength,page=1]{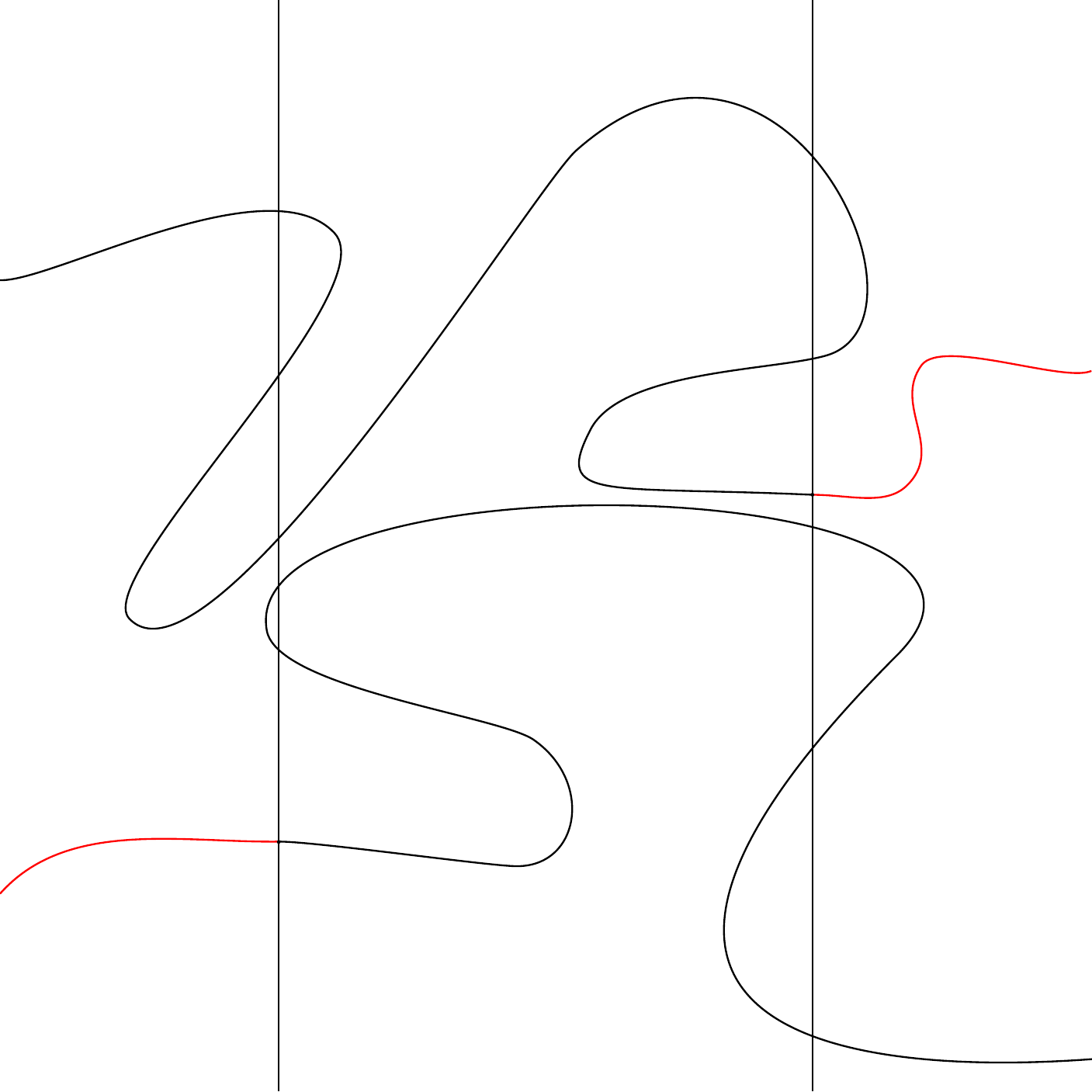}}%
    \put(0.24526043,0.24291667){\makebox(0,0)[rt]{\lineheight{1.25}\smash{\begin{tabular}[t]{r}$o$\end{tabular}}}}%
    \put(0.75591682,0.57062904){\makebox(0,0)[lt]{\lineheight{1.25}\smash{\begin{tabular}[t]{l}$b$\end{tabular}}}}%
    \put(0,0){\includegraphics[width=\unitlength,page=2]{Images/corner_surgery.pdf}}%
  \end{picture}%
\endgroup%

    \end{subfigure}
    \hskip 3 cm
    \begin{subfigure}[t]{0.3\textwidth}
        \def\svgwidth{\textwidth}
        %% Creator: Inkscape 1.4 (e7c3feb100, 2024-10-09), www.inkscape.org
%% Accompanies image file 'corner_surgery2.pdf' (pdf, eps, ps)
\begingroup%
  \makeatletter%
  \providecommand\color[2][]{%
    \errmessage{(Inkscape) Color is used for the text in Inkscape, but the package 'color.sty' is not loaded}%
    \renewcommand\color[2][]{}%
  }%
  \providecommand\transparent[1]{%
    \errmessage{(Inkscape) Transparency is used (non-zero) for the text in Inkscape, but the package 'transparent.sty' is not loaded}%
    \renewcommand\transparent[1]{}%
  }%
  \providecommand\rotatebox[2]{#2}%
  \newcommand*\fsize{\dimexpr\f@size pt\relax}%
  \newcommand*\lineheight[1]{\fontsize{\fsize}{#1\fsize}\selectfont}%
  \ifx\svgwidth\undefined%
    \setlength{\unitlength}{675bp}%
    \ifx\svgscale\undefined%
      \relax%
    \else%
      \setlength{\unitlength}{\unitlength * \real{\svgscale}}%
    \fi%
  \else%
    \setlength{\unitlength}{\svgwidth}%
  \fi%
  \global\let\svgwidth\undefined%
  \global\let\svgscale\undefined%
  \makeatother%
  \begin{picture}(1,1)%
    \lineheight{1}%
    \setlength\tabcolsep{0pt}%
    \put(0,0){\includegraphics[width=\unitlength,page=1]{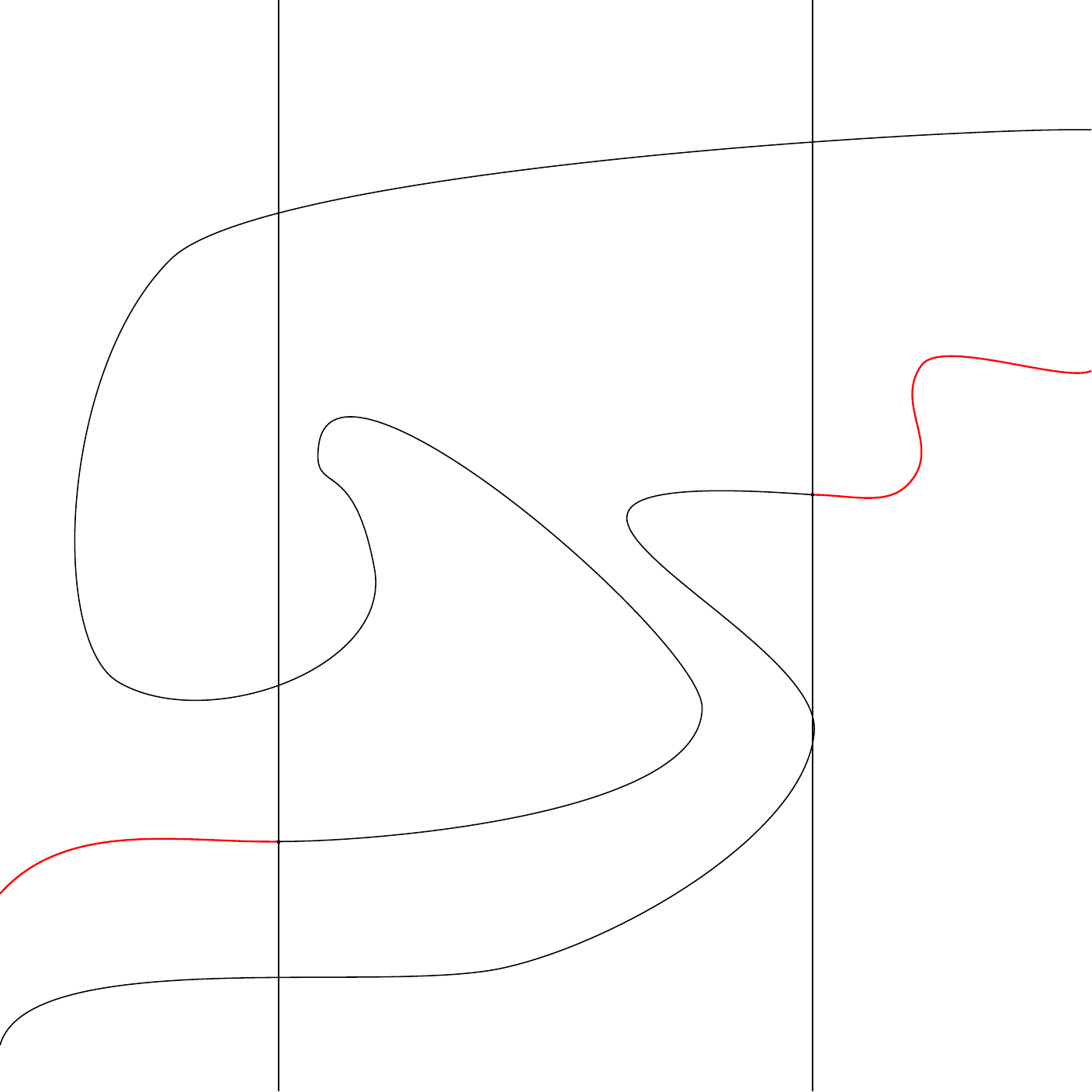}}%
    \put(0.24526043,0.24291667){\makebox(0,0)[rt]{\lineheight{1.25}\smash{\begin{tabular}[t]{r}$o$\end{tabular}}}}%
    \put(0.75591682,0.57062904){\makebox(0,0)[lt]{\lineheight{1.25}\smash{\begin{tabular}[t]{l}$b$\end{tabular}}}}%
    \put(0,0){\includegraphics[width=\unitlength,page=2]{Images/corner_surgery2.pdf}}%
  \end{picture}%
\endgroup%

    \end{subfigure}
    \caption{We can graft together the two red ends, since the vertical distance between $o$ and $b$ is smaller than the horizontal distance, and the two red paths live entirely to the left of $o$ and to the right of $b$. It does not matter if the infinite component containing $o$ is below the one containing $b$ (on the left picture), or above (on the right).}\label{fig: corner_surgery}
\end{figure}

        First, for $\omega$ to be in $E$, $b$ has to be odd according to \Cref{cor: parity is recoverable}, as the outgoing percolation edge going to the right goes in the positive direction.
        The height of the cluster of $b$ is thus given by $H(b+(1/2,-1/2))=\frac{X_{b_1}+Y_{b_2-1}}2$.
        Because $b_2>0$, $Y_{b_2-1}$ is at most $b_2-1$ and at least $-b_2+1$.
        This means that conditionally on $\omega$ being in $E$, there is a positive probability that $\xi'_0=-1$ and $Y'_{b_2-1}=-X'_{b_1}$ (where $X'$ and $Y'$ as defined like $X$ and $Y$ but on $\omega'$).
        If the event $\W$ that $\omega'$ satisfies these assertions and that $\omega$ is in $E$ is satisfied, the heights of the clusters of $o$ and $b$ are both 0 in $\omega'$, both clusters are infinite (because the configurations coincide outside of $[0,b_2]\times\Z$), and thus by the bijectivity of the height, $o$ and $b$ are in the same cluster of $\omega'$.
        Note also that $\W$ has positive probability, as it only depends on finitely many $\xi_i$s.

We can then use \Cref{lem:technical lemma for indis}.
Let $\pi_\omega$ be the random permutation of $\Z^2$ defined by $\pi_\omega(v)=v$ if $v$ is in a finite cluster, and $\pi_\omega(v)=w$ where $w$ is the second vertex following $w$ in the \textit{negative} direction if the cluster of $v$ is infinite.
Consider the random rooted network $(\Z^2,(\omega,\pi_\omega),o)$.
By including $\pi_\omega$ in the labelling, we ensure that the event $\W$ still makes sense in the context of a random rooted network (at least on the almost sure event that there are no nontrivial symmetries of $\omega$).
Let $(v_n)$ be the sequence on this random network given by the $\pi_\omega$-cycle of $o$, i.e., $v_0=o$ and $v_{n+1}=\pi_\omega(v_n)$ for any $n\in\N$.
Define $S_{\Z^2,(\omega,\pi_\omega),o}:=[0,b]\times \Z$.

Consider also the random rooted network $(\Z^2,(\omega',\pi_{\omega'}),b)$ with the sequence $(v'_n)$ given by the $\pi_{\omega'}$-cycle of $b$, and $S'_{\Z^2,(\omega',\pi_{\omega'}),b}:=[0,b]\times\Z$.

Using \Cref{lem: random permutation and SDRW are stationary}, $(v_n)$ and $(v'_n)$ are both stationary, and the two sequences are equal on $\W$ up to shifting $(v'_n)$, so we can apply \Cref{lem:technical lemma for indis} to show that $o$ in $\omega$ and $b$ in $\omega'$ have the same $\A$-type conditionally on $\W$.
We can do the same construction but starting both $(v_n)$ and $(v'_n)$ at $b$, and taking $\pi_\omega(v)$ to be the second vertex after $v$ in the \textit{positive} direction, to show that $b$ in $\omega$ and $b$ in $\omega'$ have the same $\A$-type conditionally on $\W$.
This implies that $o$ and $b$ have the same $\A$-type in $\omega$ conditionally on $\W$, contradicting its definition.
\end{proof}

\Cref{thm: corner percolation has indistinguishable clusters} allows us to prove Conjecture 5.3 in  \cite{cornerPerco2022}, a previous version of \cite{cornerPerco2025}.
The setup is the following:
we write the edges in the cluster of the origin as $\{\ldots,e_{-1},e_0,e_1\ldots\}$, with $e_{-1}$ and $e_0$ being incident to $o$. For every $i$ in $\Z$, let $\sgn(e_i)$ be defined as 1 if $e_i$ goes up or right and as $-1$ if it goes down or left (in the positive direction).

We then define the process $X$ as  $(\sgn(e_{2i}))_{i\in\Z}$ %(or $(\sgn(e_{2i+1}))_{i\in \Z}$, it does not matter) 
and $\theta$ as the shift map on $X$.
The conjecture was that conditionally on the cluster of the origin being infinite, $X$ is ergodic, i.e., that for every $X$ measurable event $E$ such that $\theta\cdot E=E$, $$\P(E\mid\text{the cluster of the origin is infinite})\in\{0,1\}.$$

This is a consequence of \Cref{thm: corner percolation has indistinguishable clusters}: if $E$ is such an event, we can define a component property by saying that a cluster $\C$ has type $\A$ if and only if the process $X$ started from some even $x$ in $\C$ (the choice of $x$ does not matter by shift invariance) satisfies $E$.
With probability 1, either all clusters have type $\A$ or none do, and by ergodicity of corner percolation, $\P(E)\in\{0,1\}$ if the cluster of the origin is infinite.

\subsection{A transient example: the Poisson zoo}
Let $\nu$ be any distribution on finite connected subgraphs of a Cayley graph $G$ containing $o$ (called lattice animals), and let $(\psi_v)_{v\in G}$  be i.i.d.\ Poisson variables with parameter $\lambda$.
Let also $(\xi_{v,n})_{v\in G,n\in \N}$ be i.i.d.\ variables with law $\nu$.

We define a site percolation $\omega$ called Poisson zoo by
$$
\omega:=\bigcup_{v\in G}\bigcup_{i=1}^{\psi_v}v\cdot\xi_{v,i}.
$$

\begin{proof}[Proof of \Cref{thm: indis for poisson zoo}]
    In order to apply \Cref{cor: technical on graphs}, we can consider $\omega$ as an $\eta$-measurable percolation, where $\eta_v:=(\psi_v,(\xi_{v,n})_{n\in\N})$.
    Let $x,y$ be two vertices, assume for contradiction that there is a positive probability for the event $E$ that $x$ and $y$ have infinite clusters, $x$ with type $\A$ and $y$ with type $\lnot\A$. 
    Take any minimal length path $S$ between $x$ and $y$.
    There are countably many possible lattice animals, and as such if $\eta_1=(\psi^1,\xi^1)$ is a sample of the Poisson zoo and $\eta_2=(\psi^2,\xi^2)$ is taken by resampling on the path $S$, there is a positive probability that $x$ has an infinite cluster of type $\A$ in $\eta_1$, that $y$ has an infinite cluster of type $\lnot\A$ in $\eta_1$, that $\psi^2_v>\psi^1_v$ for every $v$ in $S$, that $\xi^2_{v,n}=\xi^1_{v,n}$ for every $v$ in $S$ and $n\leq \psi^1_v$, and that $\psi^2_v = \psi^1_v+1$.

    This means that we simply add a lattice animal at each site along $S$, which connects in $\eta_2$ the clusters of $x$ and $y$.
    We can thus apply \Cref{cor: technical on graphs} to show as before that conditionally on $E$, $y$ in $\eta_1$ and $\eta_2$ has the same $\A$-type, and $x$ in $\eta_1$ and $y$ in $\eta_2$ have the same $\A$-type, contradicting the assumptions on $x$ and $y$.
\end{proof}

\subsection{Indistinguishability for not essentially tail properties}\label{ss:nontail}

In this subsection, $\eta$ is once again any invariant ergodic labelling of $G$, and $\omega$ any $\eta$-measurable site or bond percolation such that $(\omega,\eta)$ is jointly invariant.

\begin{definition}
    A \textbf{tail component property} is a component property $\A$ such that if $\omega$ and $\omega'$ are two percolation configurations with environments $\eta$ and $\eta'$ such that $\eta$ and $\eta'$ differ only at finitely many sites or edges, and the clusters of $o$ in $\omega$ and $\omega'$ are infinite and have a finite symmetric difference, then either $o$ has type $\A$ in both $\omega$ and $\omega'$, or in neither.

    We call a component property $\A$ \textbf{essentially tail} if there exists a tail component property $\A'$ with $\P(\{(\eta,x)\in\A\}\bigtriangleup\{(\eta,x)\in\A'\})=0$.
\end{definition}

Examples of tail component properties are easy to construct: transience, density, etc. Not essentially tail component properties are also easy when finitely many vertices can be specified in an invariant way, but this is possible only on graphs where the mass transport principle does not hold, such as the grandparent graph (see, e.g., \cite{LPbook}). Here, for Bernoulli percolation, a nontrivial not essentially tail component property, is, for example, the number of ``oldest'' vertices in the cluster. It is less straightforward to construct not essentially tail properties without relying on special vertices. An example, inspired by \cite{benjamini2004geometry}: a cluster has type $\A_k$ if one can connect it to any other cluster by adding at most $k$ edges.
It is clearly not tail, but it still might be trivial in any percolation process on Cayley graphs, and thus might be essentially tail. And, indeed, \Cref{thm: nontail indis} implies that not essentially tail component properties do not exist in invariant models on Cayley graphs.

\begin{proof}[Proof of \Cref{thm: nontail indis}]
    Assume for contradiction the existence of a nontrivial not essentially tail component property  $\A$.  
    Using ergodicity, this means that there is almost surely both type $\A$ and type $\lnot \A$ clusters coexisting in $\eta$.
    
    Fix an arbitrary vertex $x$ of $G$.
    We will show that there exists a finite $S$ such that if $\eta'$ is taken by resampling $\eta|_S$ independently, conditionally on $\eta|_{S^c}$, then with a positive probability the clusters of $x$ in $\eta$ and $\eta'$ have finite symmetric difference, but one has type $\A$ and the other $\lnot\A$.
    We call such a set $S$ \textbf{pivotal}.
    In other words, if $\A$ is not essentially tail, one has to be able to see it locally, as tail properties are properties that are by definition not seen locally.
    This will immediately allow us to apply \Cref{cor: technical on graphs} to arrive at a contradiction: any finite pivotal set $S$ fits the setup of the proposition, and thus resampling $\eta|_S$ should not change the type of $x$.

Indeed, assume once again for contradiction that a finite pivotal set does not exist.
    Then we construct the following component property: let $(\eta,x)\in \A'$ on the event that, for every finite set $F$, for $\eta'$ obtained by resampling $\eta$ independently in $F$, conditionally on the clusters of $x$ in $\eta$ and $\eta'$ to have a finite difference, $(\eta',x)$ has probability 1 to be in $\A$. Note that $\A' \subseteq \A$ by construction. 
    
    Furthermore, $\P\big((\eta,x)\in\A' \,\big|\, (\eta,x)\in\A\big)=1$, meaning that, up to zero-measure sets, $\A\subseteq \A'$. 
    Indeed, assume that $(\eta,x)\in \A$ and there is a finite set $F$ such that for $\eta'$ obtained by resampling $\eta$ independently in $F$, conditionally on the clusters of $x$ in $\eta$ and $\eta'$ to have a finite difference, $(\eta',x)$ had probability strictly less than 1 to be in $\A$. Then $F$ would be a pivotal set.
    
    Altogether, $\P((\eta,x)\in\A\bigtriangleup\A')=0$. But it is clear from its definition that $\A'$ is a tail component property. Hence $\A$ is essentially tail, contradicting the initial assumption.
\end{proof}

\begin{remark}
    The proof of this theorem can be adapted almost without modifications to show that not essentially tail component properties do not exist for invariant ergodic random permutations on Cayley graphs, or more general invariant partitions with a transient or null-recurrent stationary sequence $(v_n)$. One could also make a generalization to unimodular random networks.
\end{remark}

\bibliographystyle{alpha}
\bibliography{biblio_bibtex}
\addcontentsline{toc}{section}{References}
\end{document}